\documentclass[11pt]{article}
\usepackage[ruled]{algorithm2e}
\usepackage{graphicx}
\usepackage{mathrsfs}
\usepackage{bbm}
\usepackage{amsfonts}
\usepackage{amssymb}
\usepackage{amsmath,amsthm}
\usepackage{amsmath}
\numberwithin{equation}{section}
\numberwithin{figure}{section}
\usepackage{indentfirst}
\usepackage{color}
\usepackage{multirow}
\usepackage{enumerate}
\usepackage{epstopdf}
\usepackage{cite,stmaryrd,txfonts}
\usepackage{tikz}
\usepackage{slashbox}
\usepackage{makecell}
\usepackage{algorithmicx}
\usepackage{booktabs}
\usepackage{subfigure}
\usepackage{authblk}
\def\opn#1#2{\def#1{\operatorname{#2}}} 
\opn\chara{char} \opn\length{\ell}
\opn\projdim{proj\,dim} \opn\injdim{inj\,dim} \opn\rank{rank}
\opn\depth{depth} \opn\grade{grade} \opn\height{height}
\opn\embdim{emb\,dim} \opn\codim{codim}

\opn\Tr{Tr} \opn\bigrank{big\,rank}
\opn\superheight{superheight}\opn\lcm{lcm}
\opn\trdeg{tr\,deg}%
\opn\reg{reg} \opn\lreg{lreg}
\opn\Ker{Ker} \opn\Coker{Coker} \opn\Im{Im} \opn\Hom{Hom}
\opn\Tor{Tor} \opn\Ext{Ext} \opn\End{End} \opn\Aut{Aut} \opn\id{id}

\opn\nat{nat}
\opn\pff{pf}
\opn\Pf{Pf} \opn\GL{GL} \opn\SL{SL} \opn\mod{mod} \opn\ord{ord}

\def\Implies{\ifmmode\Longrightarrow \else
     \unskip${}\Longrightarrow{}$\ignorespaces\fi}
\def\implies{\ifmmode\Rightarrow \else
     \unskip${}\Rightarrow{}$\ignorespaces\fi}
\def\iff{\ifmmode\Longleftrightarrow \else
     \unskip${}\Longleftrightarrow{}$\ignorespaces\fi}

\let\:=\colon
\renewcommand{\theequation}{\thesection.\arabic{equation}}
\newtheorem{Theorem}{Theorem}[section]
\newtheorem{Lemma}[Theorem]{Lemma}

\newtheorem{Remark}[Theorem]{Remark}

\newtheorem{Example}[Theorem]{Example}

\newtheorem{Assumption}[Theorem]{Assumption}

\let\epsilon=\varepsilon
\let\kappa=\varkappa
\opn\ini{in} \opn\inm{inm} \opn\Sym{Sym} \opn\diag{diag}
\opn\Ii{(i)} \opn\Iii{(ii)}

\begin{document}

\title{Adaptive finite element approximation of bilinear optimal control with fractional Laplacian}

\author[a]{ Fangyuan Wang}
\author[b]{ Qiming Wang}
\author[a]{ Zhaojie Zhou\thanks{Corresponding author, Zhaojie Zhou, Email:
zhouzhaojie@sdnu.edu.cn}}
\affil[a]{School of Mathematics and Statistics, Shandong Normal University, Jinan, 250014, China}
\affil[b]{School of Mathematical Sciences, Beijing Normal University, Zhuhai, 519087, China}

\date{}

\maketitle

\begin{abstract}
We investigate the application of a posteriori error estimates to a fractional optimal control problem with pointwise control constraints. Specifically, we address a
problem in which the state equation is formulated as an integral form of the fractional Laplacian equation, with the control variable embedded within the state equation as a coefficient. We propose two distinct finite element discretization approaches for an optimal control problem. The first approach employs a fully discrete scheme where the control variable is discretized using piecewise constant functions. The second approach, a semi-discrete scheme, does not discretize the control variable.
Using the first-order optimality condition, the second-order optimality condition, and
a solution regularity analysis for the optimal control problem, we devise a posteriori error
estimates. We subsequently demonstrate the reliability and efficiency of the
proposed error estimators. Based on the established error estimates framework, an adaptive refinement strategy is developed to help achieve the optimal convergence rate. The effectiveness of the refinement strategy is verified by numerical experiments.

\end{abstract}

%
\vspace{2pc}

%

%
%
\section{Introduction}
In this paper we aim to investigate adaptive finite element approximation of the following optimal control problems governed by fractional PDEs with bilinear control:
    \begin{eqnarray}\label{object}
  \min\limits_{u\in U_{ad}} J(y,u):= \frac{1}{2} \|y-y_d\|^2_{L^2(\Omega)}
  +\frac{\lambda}{2} \|u\|^2_{L^2(\Omega)}
\end{eqnarray}
subject to
\begin{eqnarray}\label{state}\left\{ \begin{aligned}
 (-\Delta)^{s} y+uy&=f,\ &\mbox{in}\ \Omega,\\
   y&=0, \ &\mbox{on}\ \Omega^c
\end{aligned}\right.
   \end{eqnarray}
and the control constraints
\begin{eqnarray*}
U_{ad}=\Big\{v\in L^{\infty}(\Omega)| 0<a\leq v\leq b,  \ a, b\in R^+\Big\}.
  \end{eqnarray*}
Here the parameters $\lambda>0$ and $a, b \in R^+$ satisfy the condition that $ 0< a < b$. The domain $\Omega\subset R^d(d=2,3)$ is  a bounded Lipschitz domain, and $\Omega^c :=R^d\backslash\Omega$.
In the sequel, $y$ is the state and $u$ is the control variable.  The function $y_d\in L^2(\Omega)$ is referred to a desired state and   $f$ is a fixed source function in $H^{-s}(\Omega)$. The fractional Laplacian $(-\Delta)^{s}$ is defined as follows:
 \begin{eqnarray*}
  (-\Delta)^{s}y(x):=C(d,s) \ {\textrm{p.v.}}\int_{R^d} \frac{y(x)-y(w)}{| x-w|^{d+2s}}dw,\ C(d,s)=\frac{ 2^{2s}s\Gamma(s+\frac{d}{2})}{\pi^{d/2}\Gamma(1-s)},
  \end{eqnarray*}
where ''p.v.'' denotes the principal value of the integral:
\begin{eqnarray*}
{\textrm{p.v.}}\int_{R^d} \frac{y(x)-y(w)}{|x-w|^{d+2s}}dw=\lim\limits_{\epsilon\rightarrow  0}\int_{R^d\setminus B_{\epsilon}(v)} \frac{y(x)-y(w)}{|x-w|^{d+2s}}dw.
\end{eqnarray*}
Here $B_{\epsilon}(v)$ is a ball of radius $\epsilon$ centered at $x$.

PDE-constrained optimal control problems have attracted lots of attentions in the past decades following the pioneering work of J.L. Lions (\cite{lion}).
Lots of literatures are devoted to developing mathematical theory and numerical methods for different optimal control models such as distributed control problems and boundary control problems.
Among the PDE-constrained control models optimal control problem with bilinear controls became more and more popular in recent years.
Different from distributed control problems and boundary control problems bilinear controls enter the state equation as coefficients.
Therefore,  it can also be viewed as a parameter estimation problem.
This feature makes the the bilinear controls able to change some main
physical characteristics of the PDE system.

 At the same time
this feature also brings difficulty in studying the bilinear control problem mathematically and computationally.
For example,   the control variable entering in the state equation as a coefficient results in a nonlinear dependence of the solution on the control variable, which leads to a lack of uniqueness in the solution.
To the best of our knowledge, the first to provide a finite element discrete analysis for elliptic optimal control problems with bilinear control is \cite{k}, where a priori error estimates are derived.
Later, numerical discretization of bilinear optimal control problems constrained by various PDEs, such as elliptic equations (\cite{win,k}), convection-diffusion equations (\cite{huweiwei,borz}) and parabolic equations (\cite{sha,xu}) have
been widely studied.

Compared with  bilinear optimal control problems constrained by integer PDEs  the literature is relatively few for bilinear optimal control problems constrained by fractional PDEs.
In \cite{casas2018}  the authors provide an analysis for the existence of solutions and first as well as second order optimality conditions of a bilinear optimal control constrained by an evolution equation involving the fractional Laplace.  Using the second-order optimality conditions required for such problems, finite element discretization and a priori error estimate for optimal control problem constrained by fractional elliptic equations is studied in \cite{bifen}.
Aside from the nonlinear dependency of the state on the control variable, another key characteristic of the model (\ref{object})-(\ref{state}) is the fractional Laplacian operator, which is nonlocal and is becoming an increasingly relevant modeling tool in fields such as fluids and image denoising (\cite{con,ga}). It is worth noting that the solutions of fractional differential equations frequently feature singularities, even with smooth data input, necessitating the use of local refined meshes.
  Although adaptive mesh refinement has been found to be very useful in computational optimal control problems \cite{liu20012,liu2002,liu20031,hin,koh,gongyan}, but there is less work on   adaptive finite element methods for bilinear optimal control problems. In \cite{chang}, the optimal control problem for two-dimensional bilinear parabolic equations has been studied. The a posteriori error estimation results are verified using the adaptive finite element method. In \cite{zheng}, the authors propose an improvement on the a posteriori error estimator for the  bilinear optimal control problem constrained by elliptic equation given in \cite{liu}, giving a new definition of the error estimator and an analysis of its effectiveness.

In the present work we focus on  adaptive finite element approximatioin for the fractional bilinear optimal control problem (\ref{object})-(\ref{state}).
We discuss two ways for discretizing the optimal control issue (\ref{object})-(\ref{state}): the fully discrete approach, where the control variable is approximated by piecewise constant functions, tthe semi-discrete method, also known as variational discretization, in which the control space is not discretized. We derive residual type a posteriori error estimate to drive the adaptive mesh refinement and
 achieve higher accuracy at lower computational cost.
One of the difficulties is the nature of the residual, which  is not necessarily in
$L^2(\Omega)$. We refer to \cite{Fau} and introduce weighted residuals. With the help of the second-order sufficient optimality conditions, the Scott-Zhang operator and  the inverse estimate for the fractional Laplacian
we obtain the reliable and efficient analysis of a posteriori error estimates.
 With D\"{o}rfler's  marking criterion, an $h$ adaptive method driven by the a posterior error estimator is described. Finally, numerical examples are provided to validate the theoretical conclusions.

The organization of this paper is as follows. In Section 2, we introduce fractional Sobolev spaces, discretization of piecewise polynomials and some known facts useful for our purposes. In Section 3, we present some regularity and a posteriori error estimation results for the state equation. In Section 4, we first introduce the first-order and second-order optimality conditions, and then design two finite element schemes. In Section 5, we present the posterior error estimator and proved the reliability and efficiency of the estimator. In Section 6, we design an adaptive strategy that provides the best experimental convergence rate for the numerical examples presented.
\section{Preliminaries}
In this section, we introduce some preliminaries about fractional Sobolev spaces and recall some facts that will be used later.
\subsection{Sobolev spaces}
Let $C$ denote a generic constant, which is independent of the mesh parameters $h$ but that might depend on $s,\ d$ and $\Omega$. $C$'s value may change with each occurrence.
 For $s\geq0,$ we let space $H^s(\Omega)$ denote the Sobolev space of order $s$ endowed with the norm $\|\cdot\|_{H^s(\Omega)}$ and seminorm $|\cdot|_{H^s(\Omega)}$. We denote by $\widetilde{H}^s(\Omega)$ be the subspace of $H^s(\Omega)$ consisting of functions whose extension by zero to $R^d$ is in $H^s(R^d)$. For $0\leq s<1/2$, it is well known that $\widetilde{H}^s(\Omega)$ coincides with $H^s(\Omega)$. For $s\geq1/2$, we define $\widetilde{H}^s(\Omega)$ as the closure of $C^{\infty}_0(\Omega)$ in $H^s(R^d)$.
  Let $(\cdot,\cdot)$ denote the inner product of $L^2(\Omega)$ and $\langle\cdot,\cdot\rangle$ denote the duality between $\widetilde{H}^s(\Omega)$ and its dual space $(\widetilde{H}^s(\Omega))^*=\widetilde{H}^{-s}(\Omega)$ with $s\in[0,1].$

\subsection{Discretization}
 We begin by partitioning the domain $\Omega$ into simplices $T$ with size $h_{\bullet}:=|T|^{1/d}\ (h_{\bullet}\in L^{\infty}(\Omega))$, forming a conforming partition $ T_{\bullet}= \{T\}$.  Denote by $\mathbb{T}$ the collection of conforming and shape regular meshes that refine an initial mesh $T_{\bullet0}$. Let $h_{\star}=\max\limits_{T\in T_{\bullet}}h_{\bullet} $ and $|T|=\rm{diam}(T):=\sup_{x,y\in T}|x-y|$, whereas $|T|$ denotes the surface measure.

We define $T_{\bullet}$ as being $\gamma$-shape regular if
$$\frac{\rm{diam(T)}}{|T|^{1/d}}\leq \gamma ,\ \mbox{for\ all\ elements\ }T\in T_{\bullet}.$$
 For each element $T\in T_{\bullet}$ and $k\in \mathbb{N}_0$, we define the $k$-th order element patch inductively as follows:
 \begin{align*}
  & \Omega_{\bullet}^0(T):=T, T_{\bullet}^0(T):=\{T\},\\
 &\Omega_{\bullet}^k(T):=\mathrm{interior}(\bigcup\limits_{{T'}\in  T_{\bullet}^k(T)}\overline{T'}),\ \mbox{where}\ T_{\bullet}^k(T):=\{T'\in T_{\bullet}: \overline{T'}\cap \overline{\Omega_{\bullet}^{k-1}(T)}\neq\emptyset \}.
 \end{align*}

 Given a mesh $T_{\bullet},$  let  $\mathbb{V}_{T_{\bullet}}$ be the finite element space made up of continuous piecewise linear functions over the triangulation $T_{\bullet}$
  \begin{eqnarray}\label{vt}
  \mathbb{V}_{T_{\bullet}}=\{v_{T_{\bullet}}\in C(\bar{\Omega}): v_{T_{\bullet}}| _T\in \mathbb{P}_1(T), \forall T\in T_{\bullet},\ v_{T_{\bullet}}=0\ \mbox{on}\ \partial\Omega\}.
    \end{eqnarray}
    \subsection{Mesh-refinement}
    In the context of regular triangulations denoted as $T_{\circ},\ T_{\bullet}$ over the domain $\Omega$, we introduce the notation $T_{\bullet}\in \rm{refine(T_{\circ})}$, signifying that $T_{\bullet}$ represents an arbitrary refinement of $T_{\circ}$, if
$$T=\cup\{T'\in T_{\bullet}:T'\subseteq T\},\ \mbox{for\ all}\ T\in T_{\circ}$$
and
$$|T'|\leq |T|/2, \mbox{for\ all }\ T\in T_{\circ}\backslash T_{\bullet}\ \mbox{and\ all}\ T'\in T_{\bullet}\ \mbox{with}\ T'\subseteq T.$$
Here $T_{\circ}\backslash T_{\bullet}$ designates the collection of refined elements, while $T_{\bullet}\backslash T_{\circ}$ comprises their corresponding successors.

The fundamental assumptions of this framework dictate that $\cup(T_{\circ}\backslash T_{\bullet}) =\cup(T_{\bullet}\backslash T_{\circ})$. Additionally, mesh parameter $h_{\bullet}$ satisfies
 \begin{align*}
h_{\bullet}\leq 2^{-1/d}h_{\circ},\ \mbox{if}\ \cup(T_{\circ}\backslash T_{\bullet}), \ \ h_{\bullet}=h_{\circ},\ \mbox{if}\  \cup(T_{\circ}\cap T_{\bullet}).
\end{align*}

\subsection{Auxiliary results}
This subsection recounts and states certain facts that were used in the a posteriori error analysis.

For a mesh $T_{\bullet}$, we introduce the weight function \cite{Fau} defined by
$$\omega_{T_{\bullet}}(x):=\inf\limits_{T\in T_{\bullet}}\inf\limits_{y\in\partial T}|x-y|.$$
We define the local mesh width function
 $$
  \widetilde{h}^{s}_{\bullet}=
 h^{s}_{\bullet},\ \mbox{if}\ s\in(0,\frac{1}{2}],\ \ \ \widetilde{h}^{s}_{\bullet}=
h^{s-\beta}_{\bullet}\omega_{T_{\bullet}}^{\beta},\ \mbox{if}\ s\in(\frac{1}{2},1),\ \beta:=s-{\frac{1}{2}}.
$$
 Let $\mathcal{N}_{\bullet}$ be the set of nodes of $T_{\bullet}$. For each $z\in \mathcal{N}_{\bullet}$,
choose an arbitrary element $T_z\in T_{\bullet}$ with $z\in T_z$. Let $\phi_z\in \mathbb{V}_{T_{\bullet}}$ denote the Lagrange
basis function associated with $z$, and let $\phi_z^{\bullet}\in P^{1}(T_z)$ be such that $\int_{T_z}\phi_z^{\bullet}\phi_{z'}dx=\delta_{zz'},$ for all $z'\in\mathcal{N}_{\bullet}$.
 Then, the Scott-Zhang operator (\cite{Fau}) $\Pi_{\bullet}$ defined by
$$\Pi_{\bullet}v=\sum\limits_{z\in\mathcal{N}_{\bullet}}(\int_{T_z}\phi_z^{\bullet}vdx)\phi_z$$
 satisfies the following properties.
\begin{Lemma}
 For all $0\leq s\leq t\leq1$, $\Pi_{\bullet}:L^2(\Omega)\rightarrow \mathbb{V}_{T_{\bullet}}$ is a well-defined linear projection, i.e.,
\begin{align}
\Pi_{\bullet}v_{\bullet}=v_{\bullet},\ \forall v_{\bullet}\in  \mathbb{V}_{T_{\bullet}}.\label{sz1}
   \end{align}
   Moreover, $\Pi_{\bullet}$ is stable in $\widetilde{H}^{s}(\Omega)$
     \begin{align}
\|\Pi_{\bullet}v\|_{\widetilde{H}^{s}(\Omega)}\leq C \|v\|_{\widetilde{H}^{s}(\Omega)}, \forall v\in \widetilde{H}^{s}(\Omega) \label{sz2}
   \end{align}
 and has a local first-order approximation property
  \begin{align}
\|\widetilde{h}^{-s}_{\bullet}(1-\Pi_{\bullet})v\|\leq C \|v\|_{\widetilde{H}^{s}(\Omega)}, \forall v\in \widetilde{H}^{s}(\Omega). \label{sz3}
   \end{align}
For all $v\in \widetilde{H}^{t}(\Omega)$ and $T\in T_{\bullet}$, let $\overline{\Omega_{\bullet}^k(T)}\cap \partial\Omega=\emptyset$, then it holds that
 \begin{align}\label{sz4}
\|(1-\Pi_{\bullet})v\|_{H^s(\Omega_{\bullet}^k(T))}\leq \left\{ \begin{aligned}
&C h_{\bullet}^{t-s}|v|_{H^t(\Omega_{\bullet}^{k+1}(T))},\ \mbox{if}\ t>1/2,\\
&C h_{\bullet}^{t-s}|v|_{H^t(\Omega_{\bullet}^{k+1}(T))}+\|dist(\cdot,\partial\Omega)^{-t}v\|_{L^2(\Omega_{\bullet}^{k+1}(T))},\ \mbox{if}\ t\leq1/2.
\end{aligned}\right.
   \end{align}
\end{Lemma}
The following Lemma establishes the required inverse estimate for the fractional Laplacian.
\begin{Lemma}
 Let $0<s\leq1/2$ with $s\neq1/4$, and for all $v\in\widetilde{H}^{2s}(\Omega)$, it holds that
  \begin{align}\label{inv1}
\|h^s_{\bullet}(-\Delta )^sv\|_{L^2(\Omega)}\leq C\left(\|v\|^2_{\widetilde{H}^{s}(\Omega)}+\sum\limits_{T\in T_{\bullet}}h_{\bullet}^{2s}\|v\|_{H^{2s}(\Omega^2_{\bullet}(T))}\right)^{1/2}.
   \end{align}
  Let $\overline{\Omega^1_{\bullet}(T)}\cap \partial\Omega=\emptyset$. If the right-hand side is augmented by $C B(v) $, where $B (v) $ is given by
  $$B^2(v)=\sum\limits_{T\in T_{\bullet}}h_{\bullet}^{1/2}\|dist(\cdot,\partial\Omega)^{-1/2}v\|^2_{L^{2}(\Omega^2_{\bullet}(T))},$$
  then  (\ref {inv1}) holds for $s =1/4$.
   In particular, for $0<s<1$ and all $v_{\bullet}\in \mathbb{V}_{T_{\bullet}}$,  it holds that
    \begin{align}\label{inv2}
\|\widetilde{h}^s_{\bullet}(-\Delta )^sv_{\bullet}\|_{L^2(\Omega)}\leq C\|v_{\bullet}\|^2_{\widetilde{H}^{s}(\Omega)}.
   \end{align}
\end{Lemma}
\section{The state equation}
In this section, we introduce the weak formula of (\ref{state}) and briefly review the regularity of the relevant solutions and the a posterior error estimate.

 The weak formulation of state equation (\ref{state}) reads: Find $y\in \widetilde{H}^{s}(\Omega)$ such that
\begin{eqnarray}\label{weak_state_eq1}
a(y,v)+(uy,v)=(f,v), \ \ \forall v\in \widetilde{H}^{s}(\Omega).
   \end{eqnarray}
   Here $$a(y,v)=\frac{C(d,s) }{2}\int\int_{R^d\times R^d} \frac{(y(x)-y(w))(v(x)-v(w))}{|x-w|^{d+2s}}dxdw.$$
We designate $\|\cdot\|_{\widetilde{H}^s(\Omega)}$ as the norm induced by the bilinear form $a(\cdot,\cdot)$, which is just a multiple of the $H^s(R^d)$ seminorm:
$$
\|y\|_{\widetilde{H}^{s}(\Omega)}:=\sqrt{a(y,y)}=\sqrt{\frac{C(d,s) }{2}}|y|_{H^s(R^d)}.
$$
As the $H^{s}(R^d)$ seminorm is equivalent to the $H^{s}(R^d)$ norm on $ \widetilde{H}^{s}(\Omega)$ (see \cite{Acosta2}), the Lax-Milgram lemma implies that equation (\ref{weak_state_eq1}) is well-posed. We have the following stability estimate in particular.
$$\|y\|_{\widetilde{H}^{s}(\Omega)}\leq C\|f\|_{{H}^{-s}(\Omega)}.$$
 \subsection{Regularity estimate}
To obtain a posteriori error estimate for appropriate finite element discretizations of problem (\ref{weak_state_eq1}), it is crucial to have a profound understanding of the regularity characteristics of the solution to (\ref{weak_state_eq1}).  We now present regularity results on the state equation, which are crucial to obtain regularity estimates for the optimal control variable.
\begin{Lemma}(\cite{bor})\label{regy}
Let $\Omega$ be a bound Lipchitz domain and $s\in(0,1)$, if $f(x)\in L^2(\Omega)$ and $u\in U_{ad}$, there exists a solution $y$ of (\ref{weak_state_eq1}) belongs to ${H}^{s+\sigma-\epsilon}(\Omega)$, where  $\sigma=\min\{s,\frac{1}{2}\}$ and $0<\epsilon <s$. In addition, it holds that
\begin{eqnarray*}
\|y\|_{{H}^{s+\sigma-\epsilon}(\Omega)}\leq C\epsilon^{-\tau}\|f\|_{L^2(\Omega)}.
\end{eqnarray*}
 Here $\tau=s$ for $\frac{1}{2}<s<1$ and $\tau=s+\zeta$ for $0<s\leq\frac{1}{2}$ as well as a constant $\zeta$ depending on $\Omega$ and $d$.
\end{Lemma}
By assuming a higher integrability condition on $f$, it is possible to derive an $L^{\infty}(\Omega)$-regularity result for the solution $y$ of problem  (\ref{state}).
\begin{Lemma}(\cite{ota})\label{state_regularity2}
Let $s\in(0,1),$\ $r>d/2s$, assume that $f(x)\in L^r(\Omega)$, there exists a solution $y\in  \widetilde{H}^{s}(\Omega)\cap L^{\infty}(\Omega)$ satisfies
\begin{eqnarray*}
\|y\|_{\widetilde{H}^{s}(\Omega)}+\|y\|_{L^{\infty}(\Omega)}\leq C\|f\|_{L^r(\Omega)}.
\end{eqnarray*}
\end{Lemma}

\subsection{An a posteriori error estimate for the state equation}
According to the definition of (\ref{vt}), we introduce an approximation of the solution to (\ref{weak_state_eq1})
 \begin{eqnarray}\label{weak_state2}
a(y_{T_{\bullet}}(u),v_{T_{\bullet}})+(uy_{T_{\bullet}}(u),v_{T_{\bullet}})=(f,v_{T_{\bullet}}), \ \ \forall v_{T_{\bullet}}\in \mathbb{V}_{T_{\bullet}}.
   \end{eqnarray}
    The existence and uniqueness of a solution $ y_{T_{\bullet}}(u)\in \mathbb{V}_{T_{\bullet}}$ follows from the Lax-Milgram lemma.

We further introduce the following weighted residual error as follows
 \begin{eqnarray*}
\mathcal{\eta}^2_{\mathcal{T}}(y_{T_{\bullet}}(u)):=\|\widetilde{h}^{s}_{\bullet}(f-uy_{T_{\bullet}}(u)-(-\Delta )^sy_{T_{\bullet}}(u))\|^2_{L^2(\Omega)}.
  \end{eqnarray*}
Since the $L^2$-norm is local, the error estimator can be expressed as the sum of local contributions.
$$\mathcal{\eta}^2_{\mathcal{T}}(y_{T_{\bullet}}(u)):=\sum\limits_{T\in T_{\bullet}}\mathcal{\eta}^2_{\mathcal{T}}(T,y_{T_{\bullet}}(u)),\ \mathcal{\eta}^2_{\mathcal{T}}(T,y_{T_{\bullet}}(u)):=\|\widetilde{h}^{s}_{\bullet}(f-uy_{T_{\bullet}}(u)-(-\Delta )^sy_{T_{\bullet}}(u))\|^2_{L^2(T)}.$$
A posteriori error estimator normally provided upper or lower bounds for the error:
\begin{Lemma}\label{errorstate}
Let $y\in \widetilde{H}^{s}(\Omega)$ be the solution to (\ref{weak_state_eq1}) and $ y_{T_{\bullet}}(u)\in \mathbb{V}_{T_{\bullet}}$ be its
finite element approximation obtained in (\ref{weak_state2}). Then we have
\begin{eqnarray*}
\|y-y_{T_{\bullet}}(u)\|_{\widetilde{H}^{s}(\Omega)}\leq  C_{\mathrm{yrel}} \mathcal{\eta}_{\mathcal{T}}(y_{T_{\bullet}}(u)).
 \end{eqnarray*}
 Moreover,  for $0<s\leq 1/2$ and $y\in H^{s+1/2-\epsilon}(\Omega)\cap\widetilde{H}^{s}(\Omega), 0\leq \epsilon< \min\{s, 1/2-s\}$, the lower bound of the finite element approximation error is as follows:
  \begin{eqnarray*}
\mathcal{\eta}^2_{\mathcal{T}}(y_{T_{\bullet}}(u))\leq  { C}_{\mathrm{yeff}}  \Big(\|y-y_{T_{\bullet}}(u)\|_{\widetilde{H}^{s}(\Omega)}^2+\sum\limits_{T\in T_{\bullet}}h_{\bullet}^{1-2\epsilon}\|y-y_{T_{\bullet}}(u)\|_{H^{s+1/2-\epsilon}(\Omega^3_{\bullet}(T))}^2\|\Big).
 \end{eqnarray*}
\end{Lemma}
\begin{proof}
To prove the upper bound of a posteriori error, we first observe that $y-y_{T_{\bullet}}(u)\in \widetilde{H}^{s}(\Omega)$ solves
\begin{eqnarray*}
a(y-y_{T_{\bullet}}(u),v)+(u(y-y_{T_{\bullet}}(u)),v)=\langle f-uy_{T_{\bullet}}(u)-(-\Delta )^sy_{T_{\bullet}}(u),v\rangle.
 \end{eqnarray*}
We invoke Galerkin orthogonality to arrive at
\begin{eqnarray*}
a(y-y_{T_{\bullet}}(u),v)+(u(y-y_{T_{\bullet}}(u)),v)=\langle f-uy_{T_{\bullet}}(u)-(-\Delta )^sy_{T_{\bullet}}(u),v-\Pi_{\bullet}v\rangle.
 \end{eqnarray*}
 Setting $v=y-y_{T_{\bullet}}(u)$, we utilize the Cauchy-Schwarz inequality and (\ref{sz3}) to obtain
\begin{align*}
\|y-y_{T_{\bullet}}(u)\|^2_{\widetilde{H}^{s}(\Omega)}&\leq\|\widetilde{h}^{s}_{\bullet}(f-uy_{T_{\bullet}}(u)-(-\Delta )^sy_{T_{\bullet}}(u))\|_{L^2(\Omega)}\|\widetilde{h}^{-s}_{\bullet}(1-\Pi_{\bullet})(y-y_{T_{\bullet}}(u))\|_{L^2(\Omega)}\\
&\leq \mathcal{\eta}_{\mathcal{T}}(y_{T_{\bullet}}(u))\|y-y_{T_{\bullet}}(u)\|_{\widetilde{H}^{s}(\Omega)}.
 \end{align*}
 Consequently,
 \begin{eqnarray*}
\|y-y_{T_{\bullet}}(u)\|_{\widetilde{H}^{s}(\Omega)}\leq  C \mathcal{\eta}_{\mathcal{T}}(y_{T_{\bullet}}(u)).
 \end{eqnarray*}
We use the inverse estimates (\ref{inv1})-(\ref{inv2}) to demonstrate the weak efficiency. Noting that for $0<s\leq 1/2$ with $s\neq1/4$ we have
 \begin{align*}
 \mathcal{\eta}^2_{\mathcal{T}}(y_{T_{\bullet}}(u))&=\|\widetilde{h}^{s}_{\bullet}(f-uy_{T_{\bullet}}(u)-(-\Delta )^sy_{T_{\bullet}}(u))\|^2_{L^2(\Omega)}\\
 &=\|\widetilde{h}^{s}_{\bullet}((-\Delta )^sy+uy-uy_{T_{\bullet}}(u)-(-\Delta )^sy_{T_{\bullet}}(u))\|^2_{L^2(\Omega)}\\
 &=\|h^{s}_{\bullet}(-\Delta )^s(y-y_{T_{\bullet}}(u))\|^2_{L^2(\Omega)}+\|h^{s}_{\bullet}u(y-y_{T_{\bullet}}(u))\|^2_{L^2(\Omega)}\\
 &\leq\|h^{s}_{\bullet}(-\Delta )^s\Pi_{\bullet}(y-y_{T_{\bullet}}(u))\|^2_{L^2(\Omega)}+\|h^{s}_{\bullet}(-\Delta )^s(1-\Pi_{\bullet})(y-y_{T_{\bullet}}(u))\|^2_{L^2(\Omega)}+C\|u\|\ \|y-y_{T_{\bullet}}(u)\|\\
 &\leq \|\Pi_{\bullet}(y-y_{T_{\bullet}}(u))\|^2_{\widetilde{H}^{s}(\Omega)}+ \|(1-\Pi_{\bullet})(y-y_{T_{\bullet}}(u))\|^2_{\widetilde{H}^{s}(\Omega)} \\ &\quad+\sum\limits_{T\in T_{\bullet}}h_{\bullet}^{2s}\|(1-\Pi_{\bullet})(y-y_{T_{\bullet}}(u))\|_{H^{2s}(\Omega^2_{\bullet}(T))}^2
+C\|y-y_{T_{\bullet}}(u)\|_{\widetilde{H}^{s}(\Omega)}.
 \end{align*}
 By (\ref{sz2}), we can obtain
  \begin{align*}
  \|\Pi_{\bullet}(y-y_{T_{\bullet}}(u))\|^2_{\widetilde{H}^{s}(\Omega)}+ \|(1-\Pi_{\bullet})(y-y_{T_{\bullet}}(u))\|^2_{\widetilde{H}^{s}(\Omega)}\leq C\|y-y_{T_{\bullet}}(u)\|^2_{\widetilde{H}^{s}(\Omega)}.
  \end{align*}
 We observe that  $0<s\leq 1/2$ and $y-y_{T_{\bullet}}(u)\in H^{s+1/2-\epsilon}(\Omega)\cap\widetilde{H}^{s}(\Omega), 0\leq \epsilon< \min\{s, 1/2-s\}$, it is noted in (\ref{sz4}) that
   \begin{align*}
h_{\bullet}^{2s}\|(1-\Pi_{\bullet})(y-y_{T_{\bullet}}(u))\|_{H^{2s}(\Omega^2_{\bullet}(T))}^2\leq Ch_{\bullet}^{1-2\epsilon}\|y-y_{T_{\bullet}}(u)\|_{H^{s +1/2-\epsilon}(\Omega^3_{\bullet}(T))}^2.
  \end{align*}
We finally conclude that
   \begin{align*}
\mathcal{\eta}^2_{\mathcal{T}}(y_{T_{\bullet}}(u))\leq  { C}_{\mathrm{eff}}  \Big(\|y-y_{T_{\bullet}}(u)\|_{\widetilde{H}^{s}(\Omega)}^2+\sum\limits_{T\in T_{\bullet}}h_{\bullet}^{1-2\epsilon}\|y-y_{T_{\bullet}}(u)\|_{H^{s+1/2-\epsilon}(\Omega^3_{\bullet}(T))}^2\|\Big).
    \end{align*}
\end{proof}

\section{The optimal control problem}
The weak formulation of the optimal control problem (\ref{object})-(\ref{state}) can be stated as follows:
 \begin{align}\label{weak_object}
  \min\limits_{y\in \widetilde{H}^{s}(\Omega),\ u\in  U_{ad}} J(y,u)
\end{align}
subject to
\begin{align}\label{weak_state}
a(y,v)+(uy,v)=(f,v), \ \ \forall v\in \widetilde{H}^{s}(\Omega).
   \end{align}
Here $
U_{ad}=\Big\{v\in L^{\infty}(\Omega)| 0<a\leq v\leq b,  \ a, b\in R^+\Big\}.
$
According to \cite{bifen}, we know that the optimal control problem (\ref{weak_object})-(\ref{weak_state}) admits at least one global solution $(y,u)\in \widetilde{H}^{s}(\Omega)\times U_{ad}.$
We import the control to state map $S: L^{\infty}(\Omega)\rightarrow\widetilde{H}^{s}(\Omega)$, which given a control $u$, associated to it the unique state $y=Su$.

 In the context of the control problem (\ref{weak_object})-(\ref{weak_state}), which is not convex, we analyze optimality conditions based on local solutions in $L^2(\Omega)$. By the definition of $S$, we introduce the reduced cost functional $j: U_{ad}\rightarrow R$:
$$j(u)=J(Su,u):= \frac{1}{2} \|Su-y_d\|^2_{L^2(\Omega)}
  +\frac{\lambda}{2} \|u\|^2_{L^2(\Omega)}.$$
  We know that $u\in U_{ad}$ is called an optimal control for problem if $j(u)=\min\limits_{\bar{u}\in U_{ad}} j(\bar{u})$, then  $y=Su\in \widetilde{H}^{s}(\Omega)$ is called the optimal state associated with $u$. To provide first-order optimality conditions, we give the equation
\begin{eqnarray}\label{adj}
a(w,z)+(uz,w)=(y-y_d,w), \ \ \forall w\in \widetilde{H}^{s}(\Omega),
   \end{eqnarray}
 which  is called the adjoint equation, and its solution $z\in \widetilde{H}^{s}(\Omega)$ is called the adjoint state associated with $(y,u)$. The following Theorem presents the necessary optimality condition for the optimal control problem (\ref{weak_object})-(\ref{weak_state}).
  \begin{Theorem}
  If $u\in U_{ad}$ denotes the optimal control for (\ref{weak_object})-(\ref{weak_state}), then we can obtain the variational inequality
 \begin{align}\label{j'u}
j'(u)(v-u)=(\lambda u-yz,v-u)\geq0,\ \ \forall v\in  U_{ad}.
   \end{align}
   Here $y=Su$ solves (\ref{weak_state}), $z\in \widetilde{H}^{s}(\Omega)$ solves (\ref{adj}) and   $j'(u)$ denotes the Gate$\rm{\hat{a}}$ux detervative of $j$ at $u$.
  \end{Theorem}

\subsection{First and second order optimality conditions}
 Next, we give the first-order optimality conditions \cite{bifen}.
   \begin{Theorem}
   For any locally optimal control $u\in  U_{ad}$ for problem (\ref{weak_object})-(\ref{weak_state}), and for the corresponding state $y\in \widetilde{H}^{s}(\Omega)$ and adjoint state $z\in \widetilde{H}^{s}(\Omega)$, the following equations and variational inequality hold:
   \begin{align}
&a(y,v)+(uy,v)=(f,v), \ \ \forall v\in \widetilde{H}^{s}(\Omega),\label{lianxuy}\\
&a(w,z)+(uz,w)=(y-y_d,w), \ \ \forall w\in \widetilde{H}^{s}(\Omega)\label{lianxuz}
   \end{align}
   and
   \begin{eqnarray}\label{var_con}
(\lambda u-yz,v-u)\geq0, \ \ \forall v\in  U_{ad}.
   \end{eqnarray}
    \end{Theorem}
    For $a,b\in R$,  we introduce a projection operator $\Pi_{[a, b]}: L^{\infty}(\Omega)\rightarrow U_{ad}$ defined by
 \begin{eqnarray}\label{projection operator}
 \Pi_{[a,b]}(v)=\min\{b, \max\{a,v\}\}.
  \end{eqnarray}
 We can conclude the following result: If the regularization parameter $\lambda>0,$  then the variational inequality (\ref{var_con}) is equivalent to the following projection formula
   \begin{eqnarray}\label{u}
 u=\Pi_{[a,b]}\left(\lambda^{-1}yz\right).
  \end{eqnarray}
  From the above projection formula, we can obtain an important regularity result \cite{bifen}.
  \begin{Lemma}
 Let $s\in(0, 1),\ r > d/2s,\ 0<\epsilon< s$, assume that $f,\ y_d \in L^2(\Omega)\cap L^r(\Omega)$, there exists a locally optimal solution $(y,z,u)$ for the first-order optimality conditions satisfing
    \begin{eqnarray}\label{opregy}
 y,\ z\in H^{s+\sigma-\epsilon}(\Omega)\cap L^{\infty}(\Omega),\ \|yz \|_{ H^{s+\sigma-\epsilon}(\Omega)}\leq C \epsilon^{-\tau}
  \end{eqnarray}
and
$$
 \|u\|_{ H^{s+\sigma-\epsilon}(\Omega)}\leq C \epsilon^{-\tau},
$$
 where $\sigma=\frac{1}{2}$ for $\frac{1}{2}<s<1$ and $\sigma=s-\epsilon$ for $0<s\leq\frac{1}{2}$,\ $\tau=\frac{1}{2}$ for $\frac{1}{2}<s<1$ and $\tau=\frac{1}{2}+\zeta$ for $0<s\leq\frac{1}{2}$. Here $\zeta$ denotes positive constants that depend on $\Omega,\ d,\ u$ respectively.
  \end{Lemma}

  Define $\Lambda=\lambda u-yz$, where $(y,z,u)\in \widetilde{H}^{s}(\Omega)\times\widetilde{H}^{s}(\Omega)\times U_{ad}$ is the solution of (\ref{lianxuy})-(\ref{var_con}). For every $\kappa>0$, we introduce the following cone of critical directions \cite{casas2018}:
     \begin{eqnarray}\label{G_u}
G^{\kappa}_u=\{v(x)\in L^2(\Omega):v(x)=0\ \mbox{ if}\ |\Lambda|>\kappa\ \mbox{and}\ v(x)\ \mbox{satisfies}\ (\ref{v})\}
  \end{eqnarray}
  with
    \begin{eqnarray}\label{v}v(x)\left\{ \begin{aligned}
&\geq0,\ \mbox{if}\ u(x)=a,\\
&\leq0,\ \mbox{if}\ u(x)=b.
\end{aligned}\right.
 \end{eqnarray}
     In the upcoming Lemma \ref{auxiliary_es}, we will use the following necessary second order optimality conditions \cite{bifen}.
      \begin{Lemma}
   Let $u\in  U_{ad}$ satisfies the first-order optimality condition (\ref{var_con}). Then we have
   \begin{eqnarray}\label{second}
j''(u)(v,v)\geq \mu \|v\|^2,\ \ \ \forall v\in  G^{\kappa}_{u},
   \end{eqnarray}
   where $\mu>0$ and $ G^{\kappa}_{u}$ is given by the definition of (\ref{G_u}).
   \end{Lemma}

\subsection{Discrete scheme for the optimal control problem}
We propose two finite element discretization schemes to approximate the optimal solution to the control problem (\ref{weak_object})-(\ref{weak_state}): both techniques discretize the state and adjoint variables with piecewise linear functions. They differ in the discretization techniques of the control variable. Section 4.2.1 considers that the control variable is discretized by piecewise constant functions, i.e., fully discrete. Section 4.2.2 considers that the control variable is not discrete, i.e., semi-discrete, also so-called variational discretization methods.
\subsubsection{The fully discrete scheme}
Based on the finite element space $\mathbb{V}_{T_{\bullet}}$, we can define the finite dimensional approximation to the optimal control problem (\ref{weak_object})-(\ref{weak_state}) as follows:
$$
  \min\limits_{(y_{T_{\bullet}},u_{T_{\bullet}})\in \mathbb{V}_{T_{\bullet}}\times U_{ad,h}} J(y_{T_{\bullet}},u_{T_{\bullet}})
$$
subject to
 \begin{eqnarray}\label{lisanstate}
a(y_{T_{\bullet}},v_{T_{\bullet}})+(u_{T_{\bullet}}y_{T_{\bullet}},v_{T_{\bullet}})=(f,v_{T_{\bullet}}),\ \forall v_{T_{\bullet}}\in \mathbb{V}_{T_{\bullet}}.
  \end{eqnarray}
Here the discrete admissible set $U_{ad,h}$ is defined by
$$U_{ad,h}=U_{h}\cap U_{ad},$$
where $U_{h}=\{u_{T_{\bullet}}\in L^{\infty}(\Omega):u_{T_{\bullet}}|_T\in   \mathbb{P}_0(T), \forall T\in T_{\bullet}\}$. Given this configuration, the discretized first-order necessary optimality conditions can be stated as follows:
 \begin{eqnarray}\label{lisan}\left\{ \begin{aligned}
&a(y_{T_{\bullet}},v_{T_{\bullet}})+(u_{T_{\bullet}}y_{T_{\bullet}},v_{T_{\bullet}})=(f,v_{T_{\bullet}}), &\forall v_{T_{\bullet}}\in \mathbb{V}_{T_{\bullet}},\\
&a(w_{T_{\bullet}},z_{T_{\bullet}})+(u_{T_{\bullet}}z_{T_{\bullet}},w_{T_{\bullet}})=(y_{T_{\bullet}}-y_d,w_{T_{\bullet}}), &\forall w_{T_{\bullet}}\in \mathbb{V}_{T_{\bullet}},\\
 &(\lambda u_{T_{\bullet}}-y_{T_{\bullet}}z_{T_{\bullet}}, v_{T_{\bullet}}-u_{T_{\bullet}})\geq0,\ &\forall v_{T_{\bullet}}\in  U_{ad,h}.
  \end{aligned}\right.
  \end{eqnarray}
 Here, we call $y_{T_{\bullet}},\ z_{T_{\bullet}}$ and $u_{T_{\bullet}}$ the discrete optimal state, adjoint state and control, respectively.
  Let $\chi: U_{h}\rightarrow R$,\ $u_{T_{\bullet}}\mapsto \chi(u_{T_{\bullet}})=\int_{\Omega}|u_{T_{\bullet}}|dx=\sum\limits_{ T\in T_{\bullet}}|u_{T_{\bullet}}|\cdot|T|.$ The  previous variational inequality can be rewritten as
  $$(\sum\limits_{ T\in T_{\bullet}}\lambda\cdot |T|\cdot u_{T_{\bullet}}|_{T}-\sum\limits_{ T\in T_{\bullet}}\int_{T}y_{T_{\bullet}}z_{T_{\bullet}}dx)( v_{T_{\bullet}}|_{T}-u_{T_{\bullet}}|_{T})\geq 0,\ 0<a\leq v_{T_{\bullet}}|_{T}\leq b.$$
According to \cite{Tro}, we know that if $\lambda>0,$ then the following projection formula holds
  \begin{align*}
u_{T_{\bullet}}|_T=\Pi_{a,b}((\lambda|T|)^{-1}\int_{T}y_{T_{\bullet}}z_{T_{\bullet}}dx).
  \end{align*}
In order to analyze the a posteriori error estimator for the fully discrete scheme, we first introduce the auxiliary variable  $\tilde{u}$ based on the projection operator $\Pi_{[a,b]}$ as defined in (\ref{projection operator})
  \begin{eqnarray*}
\tilde{u}=\Pi_{[a,b]}\left(\lambda^{-1}y_{T_{\bullet}}z_{T_{\bullet}}\right).
  \end{eqnarray*}
  An advantageous characteristic of the $\tilde{u}$ definition is its fulfillment of the subsequent variational inequality
 \begin{eqnarray}\label{var_au}
(\lambda \tilde{u}-y_{T_{\bullet}}z_{T_{\bullet}},v-\tilde{u})\geq0, \ \ \forall v\in  U_{ad}.
   \end{eqnarray}
   \subsubsection{The semi-discrete scheme}
In the following, we will discuss the variational discretization approach, also known as the method introduced by Hinze in \cite{Hinze}.
We propose the semi-discrete scheme as follows:
Find $(y_{T_{\bullet}},u_{T_{\bullet}})$ such that $\min J(y_{T_{\bullet}},u_{T_{\bullet}})$ s.t.
 \begin{eqnarray*}
a(y_{T_{\bullet}},v_{T_{\bullet}})+(u_{T_{\bullet}}y_{T_{\bullet}},v_{T_{\bullet}})=(f,v_{T_{\bullet}}),\ \forall v_{T_{\bullet}}\in \mathbb{V}_{T_{\bullet}}.
  \end{eqnarray*}
Here $u_{T_{\bullet}}\in U_{ad}.$
The existence of a discrete solution and first-order optimality conditions for the semi-discrete scheme follow standard arguments and are specified as follows:
 Find $(y_{T_{\bullet}},z_{T_{\bullet}},u_{T_{\bullet}})\in \mathbb{V}_{T_{\bullet}}\times \mathbb{V}_{T_{\bullet}}\times U_{ad}$ such that
 \begin{eqnarray*}\left\{ \begin{aligned}
&a(y_{T_{\bullet}},v_{T_{\bullet}})+(u_{T_{\bullet}}y_{T_{\bullet}},v_{T_{\bullet}})=(f,v_{T_{\bullet}}), &\forall v_{T_{\bullet}}\in \mathbb{V}_{T_{\bullet}},\\
&a(w_{T_{\bullet}},z_{T_{\bullet}})+(u_{T_{\bullet}}z_{T_{\bullet}},w_{T_{\bullet}})=(y_{T_{\bullet}}-y_d,w_{T_{\bullet}}), &\forall w_{T_{\bullet}}\in \mathbb{V}_{T_{\bullet}},\\
 &(\lambda u_{T_{\bullet}}-y_{T_{\bullet}}z_{T_{\bullet}}, v-u_{T_{\bullet}})\geq0,\ &\forall v\in  U_{ad}.
  \end{aligned}\right.
  \end{eqnarray*}
The projection formulas for the control variable $u_{T_{\bullet}}$ can be expressed as follows
 \begin{align*}
u_{T_{\bullet}}=\Pi_{a,b}(\lambda^{-1}y_{T_{\bullet}}z_{T_{\bullet}}).
  \end{align*}

 \section{A posteriori error analysis}
 In this section, we give the following result which is instrumental to derive the a posteriori error estimates.
      \begin{Lemma}\label{auxiliary_es}(\cite{bifen})
  Let $s\in(0,1)$, $r>d/2s$, and $f,\ y_d\in L^2(\Omega)\cap L^r(\Omega)$. Assume that $y_{T_{\bullet}}$ is uniformly bounded in $L^{d/s}(\Omega)$. Then it holds that
   \begin{eqnarray}\label{auxiliary_es}
Q\|u- \tilde{u}\|^2\leq(j'( \tilde{u})-j'(u))( \tilde{u}-u).
   \end{eqnarray}
   Here $Q=2^{-1}\min\{\mu,\lambda\}$. The constant $\mu$ is given as in (\ref{second}) and $\lambda$ is the regularization parameter.
   \end{Lemma}
   \subsection{A posteriori error analysis for the fully discrete scheme}

In this section, we devise and analyze an a posteriori error estimator for the
fully discrete scheme.
 We first define the a posteriori error indicator and estimator for the optimal control variable
 \begin{eqnarray*}
\mathcal{\eta}^2_{u,T}(u_{T_{\bullet}}):=\|\tilde{u}-u_{T_{\bullet}}\|^2,\  \mathcal{\eta}^2_{\mathcal{U}}(u_{T_{\bullet}}):=\sum\limits_{T\in T_{\bullet}}\mathcal{\eta}^2_{u,T}(u_{T_{\bullet}}).
 \end{eqnarray*}
 We further define the weighted residual error estimators for the state and adjoint state variables as follows
 \begin{eqnarray}
\mathcal{\eta}_{y,T}^2(y_{T_{\bullet}}):=\|\widetilde{h}^{s}_{\bullet}(f-u_{T_{\bullet}}y_{T_{\bullet}}-(-\Delta )^sy_{T_{\bullet}})\|^2_{L^2(T)},\label{eydefin}\
 \mathcal{\eta}^2_{\mathcal{Y}}(y_{T_{\bullet}}):=\sum\limits_{T\in T_{\bullet}}\mathcal{\eta}^2_{y,T}(y_{T_{\bullet}}),\\ \mathcal{\eta}^2_{z,T}(z_{T_{\bullet}}):=\|\widetilde{h}^{s}_{\bullet}(y_{T_{\bullet}}-y_d-u_{T_{\bullet}}z_{T_{\bullet}}-(-\Delta )^sz_{T_{\bullet}})\|^2_{L^2(T)},\ \mathcal{\eta}^2_{\mathcal{Z}}(z_{T_{\bullet}}):=\sum\limits_{T\in T_{\bullet}}\mathcal{\eta}^2_{z,T}(z_{T_{\bullet}})\label{ezdefin}.
  \end{eqnarray}
   Next, we give the following auxiliary problems
   \begin{eqnarray}\left\{ \begin{aligned}
&a(y(u_{T_{\bullet}}),v)+(u_{T_{\bullet}}y(u_{T_{\bullet}}),v)=(f,v), &\forall v\in \widetilde{H}^{s}(\Omega),\label{fuzhuy}\\
&a(w,z(y_{T_{\bullet}}))+(u_{T_{\bullet}}z(y_{T_{\bullet}}),w)=(y_{T_{\bullet}}-y_d,w), &\forall w\in \widetilde{H}^{s}(\Omega).\label{fuzhuz}
  \end{aligned}\right.
  \end{eqnarray}
We know that $y_{T_{\bullet}}$ and $z_{T_{\bullet}}$ are the finite element solution of $y(u_{T_{\bullet}})\in H^{s+1/2-\epsilon}(\Omega)\cap\widetilde{H}^{s}(\Omega)$ and $ z(y_{T_{\bullet}})\in H^{s+1/2-\epsilon}(\Omega)\cap\widetilde{H}^{s}(\Omega),$ respectively.
  According to Lemma \ref{errorstate}, the following estimates hold.
  \begin{Lemma}\label{upper-lower}
   For $0 <s< 1$, $f,\ y_d\in L^2(\Omega)$, the weighted residual error estimators are reliable:
  \begin{eqnarray*}
\|y(u_{T_{\bullet}})-y_{T_{\bullet}}\|_{\widetilde{H}^{s}(\Omega)}\leq   C_{\mathrm{yrel}}  \mathcal{\eta}_{\mathcal{Y}}(y_{T_{\bullet}}),\ \|z(y_{T_{\bullet}})-z_{T_{\bullet}}\|_{\widetilde{H}^{s}(\Omega)}\leq  C_{\mathrm{zrel}} \mathcal{\eta}_{\mathcal{Z}}(z_{T_{\bullet}}).
 \end{eqnarray*}
Moreover,  for $0<s\leq 1/2,\ 0\leq \epsilon< \min\{s, 1/2-s\}$, the estimators are also efficient
  \begin{align*}
 \mathcal{\eta}_{\mathcal{Y}}(y_{T_{\bullet}})&\leq   C_{\mathrm{yeff}} \Big(\|y(u_{T_{\bullet}})-y_{T_{\bullet}}\|_{\widetilde{H}^{s}(\Omega)}^2+\sum\limits_{T\in T_{\bullet}}h_{\bullet}^{1-2\epsilon}\|y(u_{T_{\bullet}})-y_{T_{\bullet}}\|_{H^{s+1/2-\epsilon}(\Omega^3_{\bullet}(T))}^2\Big),\\
  \mathcal{\eta}_{\mathcal{Z}}(z_{T_{\bullet}})&\leq  C_{\mathrm{zeff}} \Big(\|z(y_{T_{\bullet}})-z_{T_{\bullet}}\|_{\widetilde{H}^{s}(\Omega)}^2+\sum\limits_{T\in T_{\bullet}}h_{\bullet}^{1-2\epsilon}\|z(y_{T_{\bullet}})-z_{T_{\bullet}}\|_{H^{s+1/2-\epsilon}(\Omega^3_{\bullet}(T))}^2\Big).
 \end{align*}
 \end{Lemma}
 Finally, from the above-defined error estimators for the state variable, the adjoint variable, as well as the control variable, we introduce a posteriori error estimator for the fully discrete problem:
$$\mathcal{\eta}^2_{\mathcal{OCP}}=\sum\limits_{T\in T_{\bullet}}\mathcal{\eta}^2_{\mathcal{OCP},T},\ \mathcal{\eta}^2_{\mathcal{OCP},T}=C_{st}\mathcal{\eta}_{y,T}^2(y_{T_{\bullet}})+C_{adj} \mathcal{\eta}_{z,T}^2(z_{T_{\bullet}})+C_{ct}\mathcal{\eta}_{u,T}^2(u_{T_{\bullet}}).$$
Define the errors $e_y=y-y_{T_{\bullet}},\ e_z=z-z_{T_{\bullet}}, e_u=u-u_{T_{\bullet}},$ the vector $\mathbf{e}=(e_y,e_z,e_u)^T$, and the norm
\begin{eqnarray*}
\|\mathbf{e}\|_{\Omega}^{2}=\|e_y\|_{\widetilde{H}^{s}(\Omega)}^{2}+\|e_z\|_{\widetilde{H}^{s}(\Omega)}^{2}+\|e_u\|^2.
\end{eqnarray*}
\subsubsection{Reliability of the error estimator $\mathcal{\eta}_{\mathcal{OCP}}$}
\begin{Theorem}\label{Reliability}
 Let $(y,z,u)\in \widetilde{H}^{s+\sigma-\epsilon}(\Omega)\cap L^{\infty}(\Omega)\times \widetilde{H}^{s+\sigma-\epsilon}(\Omega)\cap L^{\infty}(\Omega)\times U_{ad}$ and $(y_{T_{\bullet}},z_{T_{\bullet}},u_{T_{\bullet}})\in \mathbb{V}_{T_{\bullet}}\times \mathbb{V}_{T_{\bullet}}\times U_{ad,h}$ be the solutions of problems (\ref{lianxuy})-(\ref{var_con}) and (\ref{lisan}), respectively. For some parameter $0\leq\epsilon<s$,\ $\sigma$ is the same as in the statement of (\ref{opregy}), the following upper bound on the posterior error holds
   \begin{eqnarray*}
\|\mathbf{e}\|_{\Omega}^{2}\leq \mathcal{\eta}_{\mathcal{OCP}}^2.
 \end{eqnarray*}

\end{Theorem}
\begin{proof}
By applying the triangle inequality and Lemma \ref{upper-lower}, we can obtain
 \begin{align}\label{y-yth}
\|y-y_{T_{\bullet}}\|^2_{\widetilde{H}^{s}(\Omega)}
&\leq C\|y-y(u_{T_{\bullet}})\|^2_{\widetilde{H}^{s}(\Omega)}+C\|y(u_{T_{\bullet}})-y_{T_{\bullet}}\|^2_{\widetilde{H}^{s}(\Omega)}\nonumber\\
&\leq C\|y-y(u_{T_{\bullet}})\|^2_{\widetilde{H}^{s}(\Omega)}+C\mathcal{\eta}^2_{\mathcal{Y}}(y_{T_{\bullet}}).
\end{align}
Moreover, we notice that $y-y(u_{T_{\bullet}})\in\widetilde{H}^{s}(\Omega)$ solves
\begin{align*}
a(y-y(u_{T_{\bullet}}),v)+(u(y-y(u_{T_{\bullet}})),v)=((u_{T_{\bullet}}-u)y(u_{T_{\bullet}}),v),\ \forall v\in \widetilde{H}^{s}(\Omega).
\end{align*}
It is deduced from (\ref{state_regularity2}) that
 \begin{align}\label{y-yuth}
\|y-y(u_{T_{\bullet}})\|_{\widetilde{H}^{s}(\Omega)}
&\leq C\|y(u_{T_{\bullet}})\|_{L^{\infty}(\Omega)}\ \|u_{T_{\bullet}}-u\| \nonumber\\
&\leq C\|f\|_{L^r(\Omega)}\|u_{T_{\bullet}}-u\|\leq C\|u_{T_{\bullet}}-u\|.
\end{align}
This estimate combined with (\ref{y-yth}) imply that
 \begin{align}\label{y-yh1}
\|y-y_{T_{\bullet}}\|_{\widetilde{H}^{s}(\Omega)}
\leq C\|u_{T_{\bullet}}-u\|+C\mathcal{\eta}^2_{\mathcal{Y}}(y_{T_{\bullet}}).
\end{align}

In a similar way, we can obtain that
  \begin{align}\label{z-zth}
\|z-z_{T_{\bullet}}\|^2_{\widetilde{H}^{s}(\Omega)}&\leq C\|z-z(y_{T_{\bullet}})\|^2_{\widetilde{H}^{s}(\Omega)}+C\|z(y_{T_{\bullet}})-z_{T_{\bullet}}\|^2_{\widetilde{H}^{s}(\Omega)}\nonumber\\
&\leq C\|z-z(y_{T_{\bullet}})\|_{\widetilde{H}^{s}(\Omega)}^2+C\mathcal{\eta}^2_{\mathcal{Z}}(z_{T_{\bullet}}).
\end{align}
To estimate $\|z-z(y_{T_{\bullet}})\|_{\widetilde{H}^{s}(\Omega)}$, we observe that $z-z(y_{T_{\bullet}})\in \widetilde{H}^{s}(\Omega)$ solves
  \begin{align*}
a(w,z-z(y_{T_{\bullet}}))+(u(z-z(y_{T_{\bullet}})),w)=((u_{T_{\bullet}}-u)z(y_{T_{\bullet}}),w)+(y-y_{T_{\bullet}},w)
\end{align*}
for all $w\in \widetilde{H}^{s}(\Omega).$
By assumption, $y_{T_{\bullet}}$ is uniformly bounded in $L^r(\Omega)\ (r>d/2s)$, we obtain $\|z(y_{T_{\bullet}})\|_{L^{\infty}(\Omega)}\leq C(\|y_{T_{\bullet}}\|_{L^{r}(\Omega)}+\|y_d\|_{L^r(\Omega)}) $.
Thus we have
 \begin{align}\label{z-zuth}
\|z-z(y_{T_{\bullet}})\|_{\widetilde{H}^{s}(\Omega)}&\leq \|z(y_{T_{\bullet}})\|_{L^{\infty}(\Omega)}\ \|u- u_{T_{\bullet}}\|+\|y- y_{T_{\bullet}}\|\nonumber\\
&\leq C\|u- u_{T_{\bullet}}\|+C\|y- y_{T_{\bullet}}\|_{\widetilde{H}^{s}(\Omega)}.
    \end{align}

The previous estimate and (\ref{z-zth}) allow us to deduce the a posteriori error estimate
\begin{align*}
\|z-z_{T_{\bullet}}\|_{\widetilde{H}^{s}(\Omega)}
\leq C\|u_{T_{\bullet}}-u\|+C\|y-y_{T_{\bullet}}\|_{\widetilde{H}^{s}(\Omega)}+C\mathcal{\eta}^2_{\mathcal{Z}}(z_{T_{\bullet}}).
     \end{align*}

The next goal is to estimate the error $\|u-u_{T_{\bullet}}\|$.  To accomplish
this task, we invoke the auxiliary variable $\tilde{u}=\Pi_{[a,b]}\left(\lambda^{-1}y_{T_{\bullet}}z_{T_{\bullet}}\right)$, then we have
 \begin{align}\label{u-uh0}
\|u-u_{T_{\bullet}}\|&\leq C\|u-\tilde{u}\|+C\|\tilde{u}-u_{T_{\bullet}}\|\nonumber\\
&\leq C\|u-\tilde{u}\|+C\mathcal{\eta}_{\mathcal{U}}(u_{T_{\bullet}}).
\end{align}
Setting $v =u$ in (\ref{j'u}) and $v =\tilde{u}$ in (\ref{var_au}), using (\ref{auxiliary_es}) and $j'(u)(\tilde{u}-u)\geq0$ we arrive at
 \begin{align*}
\frac{\mu }{2} \|u-\tilde{u}\|^2&\leq j'(\tilde{u})(\tilde{u}-u)-j'(u)(\tilde{u}-u)\leq j'(\tilde{u})(\tilde{u}-u)\\
&=(\alpha\tilde{u}-\tilde{z}\tilde{y},\tilde{u}-u)\leq (z_{T_{\bullet}}y_{T_{\bullet}}-\tilde{z}\tilde{y},\tilde{u}-u)\\
&=(z_{T_{\bullet}}y_{T_{\bullet}}-\tilde{z}y_{T_{\bullet}}+\tilde{z}y_{T_{\bullet}}-\tilde{z}\tilde{y},\tilde{u}-u)\\
&=((z_{T_{\bullet}}-\tilde{z})y_{T_{\bullet}}+\tilde{z}(y_{T_{\bullet}}-\tilde{y}),\tilde{u}-u).
\end{align*}
Here the auxiliary variables $\tilde{z},\ \tilde{y}\in \widetilde{H}^{s}(\Omega)$ are defined as follows
  \begin{eqnarray*}\left\{ \begin{aligned}
&a(\tilde{y},v)+(\tilde{u}\tilde{y},v)=(f,v), & \forall v\in \widetilde{H}^{s}(\Omega),\label{yhu}\\
&a(w,\tilde{z})+(\tilde{u}\tilde{z},w)=(\tilde{y}-y_{d},w), &\forall w\in \widetilde{H}^{s}(\Omega).\label{phy}
\end{aligned}\right.
 \end{eqnarray*}
 We now proceed to derive the previous inequality. We assume that the discrete solution $y_{T_{\bullet}}$ to problem (\ref{lisanstate}) are uniformly bounded in $L^{\infty}(\Omega)$, that is
  \begin{align*}
\|y_{T_{\bullet}}\|_{L^{\infty}(\Omega)}\leq C.
\end{align*}
In view of the Sobolev embedding $\widetilde{H}^{s}(\Omega)\hookrightarrow L^r(\Omega)$, we can  obtain
 \begin{align*}
 \|u-\tilde{u}\|&\leq \|z_{T_{\bullet}}-\tilde{z}\|\ \|y_{T_{\bullet}}\|_{L^{\infty}(\Omega)}+\|y_{T_{\bullet}}-\tilde{y}\|\ \|\tilde{z}\|_{L^{\infty}(\Omega)}\\
 &\leq C\|z_{T_{\bullet}}-\tilde{z}\|_{\widetilde{H}^{s}(\Omega)}+\|y_{T_{\bullet}}-\tilde{y}\|_{\widetilde{H}^{s}(\Omega)}\ (\|\tilde{y}\|_{L^{r}(\Omega)}+\|y_d\|_{L^{r}(\Omega)})\\
 &\leq C\|z_{T_{\bullet}}-\tilde{z}\|_{\widetilde{H}^{s}(\Omega)}+\|y_{T_{\bullet}}-\tilde{y}\|_{\widetilde{H}^{s}(\Omega)}\ (\|\tilde{y}\|_{\widetilde{H}^{s}(\Omega)}+\|y_d\|_{L^{r}(\Omega)})\\
 &\leq C\|z_{T_{\bullet}}-\tilde{z}\|_{\widetilde{H}^{s}(\Omega)}+\|y_{T_{\bullet}}-\tilde{y}\|_{\widetilde{H}^{s}(\Omega)}\ (\|f\|_{L^{r}(\Omega)}+\|y_d\|_{L^{r}(\Omega)}).
\end{align*}
We thus conclude that
\begin{align}\label{u-uaux}
 \|u-\tilde{u}\|\leq C\|z_{T_{\bullet}}-\tilde{z}\|_{\widetilde{H}^{s}(\Omega)}+C\|y_{T_{\bullet}}-\tilde{y}\|_{\widetilde{H}^{s}(\Omega)}.
\end{align}
To control the right hand side of (\ref{u-uaux}), we now proceed to estimate $\|y_{T_{\bullet}}-\tilde{y}\|_{\widetilde{H}^{s}(\Omega)}$. It is easy to see that
\begin{align*}
 \|y_{T_{\bullet}}-\tilde{y}\|_{\widetilde{H}^{s}(\Omega)}= \|y_{T_{\bullet}}-y(u_{T_{\bullet}})+y(u_{T_{\bullet}})-\tilde{y}\|_{\widetilde{H}^{s}(\Omega)}\leq C\mathcal{\eta}_{\mathcal{Y}}(y_{T_{\bullet}})+\|y(u_{T_{\bullet}})-\tilde{y}\|_{\widetilde{H}^{s}(\Omega)}.
\end{align*}
To control $\|y(u_{T_{\bullet}})-\tilde{y}\|_{\widetilde{H}^{s}(\Omega)}$, we observe that
\begin{align*}
 a(\tilde{y}-y(u_{T_{\bullet}}),v)+(\tilde{u}(\tilde{y}-y(u_{T_{\bullet}})),v)=(y(u_{T_{\bullet}})(u_{T_{\bullet}}-\tilde{u}),v),\ \forall v\in \widetilde{H}^{s}(\Omega).
\end{align*}
Then the following stability estimate holds
\begin{align*}
 \|\tilde{y}-y(u_{T_{\bullet}})\|_{\widetilde{H}^{s}(\Omega)}&\leq \|y(u_{T_{\bullet}})\|_{L^{\infty}(\Omega)}\ \|u_{T_{\bullet}}-\tilde{u}\|\\
  &\leq C\|f\|_{L^r(\Omega)}\ \|u_{T_{\bullet}}-\tilde{u}\|\leq C \mathcal{\eta}_{\mathcal{U}}(u_{T_{\bullet}}).
\end{align*}
We can conclude that
\begin{align*}
\|u-\tilde{u}\|\leq C\|z_{T_{\bullet}}-\tilde{z}\|_{\widetilde{H}^{s}(\Omega)}+C\mathcal{\eta}_{\mathcal{Y}}(y_{T_{\bullet}})+C\mathcal{\eta}_{\mathcal{U}}(u_{T_{\bullet}}).
\end{align*}
Finally, we control $\|z_{T_{\bullet}}-\tilde{z}\|_{\widetilde{H}^{s}(\Omega)}$. To do this, we invoke the auxiliary adjoint $z(y_{T_{\bullet}})$
\begin{align}\label{z-zaux}
 \|z_{T_{\bullet}}-\tilde{z}\|_{\widetilde{H}^{s}(\Omega)}= \|z_{T_{\bullet}}-z(y_{T_{\bullet}})+z(y_{T_{\bullet}})-\tilde{z}\|_{\widetilde{H}^{s}(\Omega)}\leq C\mathcal{\eta}_{\mathcal{Z}}(z_{T_{\bullet}})+\|z(y_{T_{\bullet}})-\tilde{z}\|_{\widetilde{H}^{s}(\Omega)}.
\end{align}
We notice that $\tilde{z}-z(y_{T_{\bullet}})$ solves the following problem
\begin{align*}
 a(w,\tilde{z}-z(y_{T_{\bullet}}))+(\tilde{u}(\tilde{z}-z(y_{T_{\bullet}})),w)=(z(y_{T_{\bullet}})(u_{T_{\bullet}}-\tilde{u}),w)+(\tilde{y}-y_{T_{\bullet}},w),\ \forall w\in \widetilde{H}^{s}(\Omega).
\end{align*}
Using a generalized H$\rm{\ddot{o}}$lder's inequality, we can get
\begin{align*}
 \|\tilde{z}-z(y_{T_{\bullet}})\|_{\widetilde{H}^{s}(\Omega)}&\leq \|z(y_{T_{\bullet}})\|_{L^{\infty}(\Omega)}\ \|u_{T_{\bullet}}-\tilde{u}\|+\|\tilde{y}-y_{T_{\bullet}}\|\\
 &\leq \|z(y_{T_{\bullet}})\|_{L^{\infty}(\Omega)}\ \|u_{T_{\bullet}}-\tilde{u}\|+\|\tilde{y}-y_{T_{\bullet}}\|_{\widetilde{H}^{s}(\Omega)}\\
  &\leq C(\|y_{T_{\bullet}}\|_{L^r(\Omega)}+\|y_d\|_{L^r(\Omega)})\ \|u_{T_{\bullet}}-\tilde{u}\|+\|\tilde{y}-y_{T_{\bullet}}\|_{\widetilde{H}^{s}(\Omega)}\\
  &\leq C \mathcal{\eta}_{\mathcal{U}}(u_{T_{\bullet}})+C \mathcal{\eta}_{\mathcal{Y}}(y_{T_{\bullet}}).
\end{align*}
Replacing this bound into (\ref{z-zaux}) and $\|u-\tilde{u}\|$ into (\ref{u-uh0}) yields
\begin{align*}
\|u-u_{T_{\bullet}}\|\leq C \mathcal{\eta}_{\mathcal{Z}}(z_{T_{\bullet}})+C\mathcal{\eta}_{\mathcal{Y}}(y_{T_{\bullet}})+C\mathcal{\eta}_{\mathcal{U}}(u_{T_{\bullet}}).
\end{align*}
Consequently,
\begin{align}\label{y-yh}
\|y-y_{T_{\bullet}}\|_{\widetilde{H}^{s}(\Omega)}\leq C \mathcal{\eta}_{\mathcal{Z}}(z_{T_{\bullet}})+C\mathcal{\eta}_{\mathcal{Y}}(y_{T_{\bullet}})+C\mathcal{\eta}_{\mathcal{U}}(u_{T_{\bullet}})
\end{align}
and
\begin{align*}
\|z-z_{T_{\bullet}}\|_{\widetilde{H}^{s}(\Omega)}\leq C \mathcal{\eta}_{\mathcal{Z}}(z_{T_{\bullet}})+C\mathcal{\eta}_{\mathcal{Y}}(y_{T_{\bullet}})+C\mathcal{\eta}_{\mathcal{U}}(u_{T_{\bullet}}),
\end{align*}
which completes the proof.
\end{proof}

\subsubsection{Efficiency of the error estimator $\mathcal{\eta}_{\mathcal{OCP}}$}
\begin{Assumption}\label{ass}
 Let $(y,z)$ and $(y(u_{T_{\bullet}}),z(y_{T_{\bullet}}))$ are the solutions of the optimal control problem of (\ref{lianxuy})-(\ref{lianxuz}) and (\ref{fuzhuy})-(\ref{fuzhuz}), respectively. The following stability estimates hold
  \begin{align*}
&\|y(u_{T_{\bullet}})-y\|_{H^{s+1/2-\epsilon}(\Omega)}\leq C\|u- u_{T_{\bullet}}\|,\\
&\|z(y_{T_{\bullet}})-z\|_{H^{s+1/2-\epsilon}(\Omega)}\leq C\|u-u_{T_{\bullet}}\| +C\|y-y_{T_{\bullet}}\|_{\widetilde{H}^{s}(\Omega)}.
    \end{align*}
\end{Assumption}

By Lemma \ref{regy}, for $s\in(1/2,1)$, the regularity of the solution of the optimal control problem over a bounded Lipschitz domain is $y,z\in H^{s+1/2-\epsilon}(\Omega)$. This has prompted the author to introduce Assumption \ref{ass} in order to analyze the efficiency estimate.
\begin{Remark}\label{Efficiency}
Under the Assumption  \ref{ass}, we have the error estimator $\mathcal{\eta}_{\mathcal{OCP}}$ satisfied the following lower bound for $h<h_0\ll1$
 \begin{eqnarray*}
\mathcal{\eta}_{\mathcal{OCP}}^2\leq C \|\mathbf{e}\|^2 + \sum\limits_{T\in T_{\bullet}}h_{\bullet}^{1-2\epsilon} \|y-y_{T_{\bullet}}\|^2_{H^{s+1/2-\epsilon}(\Omega^3_{\bullet}(T))}+\sum\limits_{T\in T_{\bullet}}h_{\bullet}^{1-2\epsilon} \|z-z_{T_{\bullet}}\|^2_{H^{s+1/2-\epsilon}(\Omega^3_{\bullet}(T))}.
     \end{eqnarray*}
\end{Remark}
\begin{proof}
According to Lemma \ref{upper-lower}, we first estimate $\|y(u_{T_{\bullet}})-y_{T_{\bullet}}\|_{\widetilde{H}^{s}(\Omega)}$. To do this, we apply a triangle inequality to obtain
  \begin{align*}
\|y(u_{T_{\bullet}})- y_{T_{\bullet}}\|_{\widetilde{H}^{s}(\Omega)}\leq C\|y(u_{T_{\bullet}})-y\|_{\widetilde{H}^{s}(\Omega)}+C\|y-y_{T_{\bullet}}\|_{\widetilde{H}^{s}(\Omega)}
    \end{align*}
    By (\ref{y-yuth}), the following stability estimate holds
  \begin{align}\label{yut-y}
\|y(u_{T_{\bullet}})-y\|_{\widetilde{H}^{s}(\Omega)}\leq C\|u- u_{T_{\bullet}}\|.
    \end{align}
     Estimate (\ref{yut-y}) yields the desired bound
     \begin{align*}
\|y(u_{T_{\bullet}})- y_{T_{\bullet}}\|_{\widetilde{H}^{s}(\Omega)}
\leq
C \|u-u_{T_{\bullet}}\| +C\|y-y_{T_{\bullet}}\|_{\widetilde{H}^{s}(\Omega)}.
    \end{align*}
Using the Assumption \ref{ass}, we can deal with the term $\sum\limits_{T\in T_{\bullet}}h_{\bullet}^{1-2\epsilon}\|y(u_{T_{\bullet}})-y_{T_{\bullet}}\|_{H^{s+1/2-\epsilon}(\Omega^3_{\bullet}(T))}^2$ in a similar way
\begin{align*}
&\sum\limits_{T\in T_{\bullet}}h_{\bullet}^{1-2\epsilon}\|y(u_{T_{\bullet}})-y_{T_{\bullet}}\|_{H^{s+1/2-\epsilon}(\Omega^3_{\bullet}(T))}^2\\
&\leq C\sum\limits_{T\in T_{\bullet}}h_{\bullet}^{1-2\epsilon}\|y(u_{T_{\bullet}})-y\|_{H^{s+1/2-\epsilon}(\Omega^3_{\bullet}(T))}^2+C\sum\limits_{T\in T_{\bullet}}h_{\bullet}^{1-2\epsilon}\|y-y_{T_{\bullet}}\|_{H^{s+1/2-\epsilon}(\Omega^3_{\bullet}(T))}^2\\
&\leq CMh^{1-2\epsilon}\|y(u_{T_{\bullet}})-y\|_{H^{s+1/2-\epsilon}(\Omega)}^2+C\sum\limits_{T\in T_{\bullet}}h_{\bullet}^{1-2\epsilon}\|y-y_{T_{\bullet}}\|_{H^{s+1/2-\epsilon}(\Omega^3_{\bullet}(T))}^2\\
&\leq CMh^{1-2\epsilon} \|u-u_{T_{\bullet}}\|^2+C\sum\limits_{T\in T_{\bullet}}h_{\bullet}^{1-2\epsilon} \|y-y_{T_{\bullet}}\|^2_{H^{s+1/2-\epsilon}(\Omega^3_{\bullet}(T))},
\end{align*}
where $M$  stands for the maximum number of times an element $T$ can appear in the entire element patch $\Omega^3_{\bullet}(T)$.
On the basis of Lemma \ref{upper-lower} and the previous estimate, we immediately obtain the local efficiency of $\mathcal{\eta}^2_{\mathcal{Y}}(y_{T_{\bullet}})$
\begin{align*}
\mathcal{\eta}^2_{\mathcal{Y}}(y_{T_{\bullet}})
 &\leq{ C}_{\mathrm{yeff}}   \bigg\{ \|u- u_{T_{\bullet}}\|^2 + \| y -y_{T_{\bullet}}\|^2_{\widetilde{H}^{s}(\Omega)}
+\sum\limits_{T\in T_{\bullet}}h_{\bullet}^{1-2\epsilon} \|y-y_{T_{\bullet}}\|^2_{H^{s+1/2-\epsilon}(\Omega^3_{\bullet}(T))}\\
&\quad
+CMh^{1-2\epsilon}  \|u- u_{T_{\bullet}}\|^2\bigg\}
\end{align*}
 Assuming that the initial size of the mesh fulfills the following condition:  $Mh_0^{1-2\epsilon}\leq C.$ For  $h_0\ll1 $, we can obtain
\begin{align}\label{ey}
 \mathcal{\eta}^2_{\mathcal{Y}}(y_{T_{\bullet}})\leq C ( \|u- u_{T_{\bullet}}\|^2 + \|y - y_{T_{\bullet}}\|^2_{\widetilde{H}^{s}(\Omega)})  + C\sum\limits_{T\in T_{\bullet}}h_{\bullet}^{1-2\epsilon} \|y-y_{T_{\bullet}}\|^2_{H^{s+1/2-\epsilon}(\Omega^3_{\bullet}(T))}.
\end{align}
Next, we proceed with the investigation of the local efficiency properties concerning the indicator $ \mathcal{\eta}^2_{\mathcal{Z}}(z_{T_{\bullet}})$, which is defined in (\ref{ezdefin}).
According to Lemma \ref{upper-lower}, we first bound $\|z(y_{T_{\bullet}})-z_{T_{\bullet}}\|_{\widetilde{H}^{s}(\Omega)}$.
   Utilizing (\ref{z-zuth}), we can thus arrive at
   \begin{align*}
 \|z(y_{T_{\bullet}})-z_{T_{\bullet}}\|_{\widetilde{H}^{s}(\Omega)}&\leq C\|z(y_{T_{\bullet}})-z\|_{\widetilde{H}^{s}(\Omega)}+C\|z-z_{T_{\bullet}}\|_{\widetilde{H}^{s}(\Omega)}\\
 &\leq  C\|u- u_{T_{\bullet}}\|+C\|y- y_{T_{\bullet}}\|_{\widetilde{H}^{s}(\Omega)}+C\|z-z_{T_{\bullet}}\|_{\widetilde{H}^{s}(\Omega)}.
\end{align*}
We now control the term $\sum\limits_{T\in T_{\bullet}}h_{\bullet}^{1-2\epsilon}\|z(y_{T_{\bullet}})-z_{T_{\bullet}}\|_{H^{s+1/2-\epsilon}(\Omega^3_{\bullet}(T))}^2$. To complete this goal, under the Assumption \ref{ass}, we use the similar arguments to proof
 \begin{align*}
&\sum\limits_{T\in T_{\bullet}}h_{\bullet}^{1-2\epsilon}\|z(y_{T_{\bullet}})-y_{T_{\bullet}}\|_{H^{s+1/2-\epsilon}(\Omega^3_{\bullet}(T))}^2\\
&\leq C\sum\limits_{T\in T_{\bullet}}h_{\bullet}^{1-2\epsilon}\|z(y_{T_{\bullet}})-z\|_{H^{s+1/2-\epsilon}(\Omega^3_{\bullet}(T))}^2+C\sum\limits_{T\in T_{\bullet}}h_{\bullet}^{1-2\epsilon}\|z-z_{T_{\bullet}}\|_{H^{s+1/2-\epsilon}(\Omega^3_{\bullet}(T))}^2\\
&\leq CMh^{1-2\epsilon}\|z(y_{T_{\bullet}})-z\|_{H^{s+1/2-\epsilon}(\Omega)}^2+C\sum\limits_{T\in T_{\bullet}}h_{\bullet}^{1-2\epsilon}\|z-z_{T_{\bullet}}\|_{H^{s+1/2-\epsilon}(\Omega^3_{\bullet}(T))}^2\\
&\leq CMh^{1-2\epsilon}(\|u-u_{T_{\bullet}}\| +\|y-y_{T_{\bullet}}\|_{\widetilde{H}^{s}(\Omega)})+C\sum\limits_{T\in T_{\bullet}}h_{\bullet}^{1-2\epsilon} \|z-z_{T_{\bullet}}\|^2_{H^{s+1/2-\epsilon}(\Omega^3_{\bullet}(T))}.
\end{align*}
Thus we have
\begin{align}\label{ez}
 \mathcal{\eta}^2_{\mathcal{Z}}(z_{T_{\bullet}})\leq C ( \|u- u_{T_{\bullet}}\|^2+ \|y - y_{T_{\bullet}}\|^2_{\widetilde{H}^{s}(\Omega)}+\|z-z_{T_{\bullet}}\|_{\widetilde{H}^{s}(\Omega)} ) + C\sum\limits_{T\in T_{\bullet}}h_{\bullet}^{1-2\epsilon} \|z-z_{T_{\bullet}}\|^2_{H^{s+1/2-\epsilon}(\Omega^3_{\bullet}(T))}.
\end{align}
Finally, we control $ \mathcal{\eta}^2_{\mathcal{Z}}(u_{T_{\bullet}})$. By assumption, $y_{T_{\bullet}}$ is uniformly bounded in $L^{\infty}(\Omega)$. To accomplish this task, we invoke the auxiliary control $u$, defined as the solution to the problem (\ref{weak_object})-(\ref{weak_state}), to rewrite $\|\tilde{u}-u_{T_{\bullet}}\|$ as follows
  \begin{align}\label{eu}
 \mathcal{\eta}^2_{\mathcal{U}}(u_{T_{\bullet}})=\|\tilde{u}-u_{T_{\bullet}}\|&\leq\|\tilde{u}-u\|+\|u-u_{T_{\bullet}}\|\nonumber\\
&=\|\Pi_{[a,b]}(\alpha^{-1}z_{T_{\bullet}}y_{T_{\bullet}})-\Pi_{[a,b]}(\alpha^{-1}zy)\|+\|u-u_{T_{\bullet}}\|\nonumber\\
&\leq C\alpha^{-1}\|z_{T_{\bullet}}y_{T_{\bullet}}-zy\|+\|u-u_{T_{\bullet}}\|\nonumber\\
&\leq C\alpha^{-1}(\|y_{T_{\bullet}}\|_{L^{\infty}(\Omega)}\ \|z_{T_{\bullet}}-z\|+\|z\|_{L^{\infty}(\Omega)}\ \|y_{T_{\bullet}}-y\|)+C\|u-u_{T_{\bullet}}\|\nonumber\\
&\leq C( \|z_{T_{\bullet}}-z\|_{\widetilde{H}^{s}(\Omega)}+ \|y_{T_{\bullet}}-y\|_{\widetilde{H}^{s}(\Omega)})+C\|u-u_{T_{\bullet}}\|.
    \end{align}
    The bound (\ref{eu}) combined (\ref{ey}) and (\ref{ez}) imply
   \begin{align*}
\mathcal{\eta}_{\mathcal{OCP}}^2\leq C \|\mathbf{e}\| + C\sum\limits_{T\in T_{\bullet}}h_{\bullet}^{1-2\epsilon} \|y-y_{T_{\bullet}}\|^2_{H^{s+1/2-\epsilon}(\Omega^3_{\bullet}(T))}+C\sum\limits_{T\in T_{\bullet}}h_{\bullet}^{1-2\epsilon} \|z-z_{T_{\bullet}}\|^2_{H^{s+1/2-\epsilon}(\Omega^3_{\bullet}(T))}.
\end{align*}
This concludes the proof.
\end{proof}
  \subsection{A posteriori error analysis for the semi discrete scheme}

In this section, we propose and analyze the a posteriori error estimator for the semi-discrete scheme discussed in Section 5.2. Unlike the estimators introduced in Section 5.1, this new estimator does not rely on the error estimator of the control variable but only on the error estimators of the state and adjoint variables, i.e.,
\begin{align}\label{semie}
\mathcal{\eta}_{\mathcal{OCP}}=\mathcal{\eta}_{\mathcal{Y}}(y_{T_{\bullet}})+\mathcal{\eta}_{\mathcal{Z}}(z_{T_{\bullet}}).
  \end{align}
  The error is defined as
$\mathbf{\bar{e}}=(e_y,e_z,e_u)^T$.
\begin{Theorem}\label{Reliability2}
 Let $(y,z,u)\in \widetilde{H}^{s+\sigma-\epsilon}(\Omega)\cap L^{\infty}(\Omega)\times \widetilde{H}^{s+\sigma-\epsilon}(\Omega)\cap L^{\infty}(\Omega)\times U_{ad}$  and $(y_{T_{\bullet}},z_{T_{\bullet}},u_{T_{\bullet}})\in \mathbb{V}_{T_{\bullet}}\times \mathbb{V}_{T_{\bullet}}\times U_{ad}$ be the solutions of the problem (\ref{lianxuy})-(\ref{var_con}) and (\ref{lisan}), respectively. Then the following upper bound of a posteriori error holds for $h<h_0\ll1$

$$\|\mathbf{e}\|^2\leq\mathcal{\eta}^2_{\mathcal{OCP}}.$$
Here $\mathcal{\eta}^2_{\mathcal{OCP}}=\sum\limits_{T\in T_{\bullet}}\mathcal{\eta}^2_{\mathcal{OCP},T},\ \mathcal{\eta}_{\mathcal{OCP},T}^2=C_{st}\mathcal{\eta}_{y,T}^2(y_{T_{\bullet}})+C_{adj} \mathcal{\eta}_{z,T}^2(z_{T_{\bullet}})$ and the constant $\sigma$ is given as in (\ref{opregy}).
\end{Theorem}
\begin{proof}
We will prove the statement by closely following the proof developed in Theorem \ref{Efficiency}, where we used $\tilde{u}=u_{T_{\bullet}}$. For the sake of brevity, we omit the detailed steps and calculations.
\end{proof}
\begin{Remark}\label{Efficiency2}

For the proof of the efficiency estimate, it is also necessary to work under the Assumption \ref{ass}, following a similar procedure as in the proof outlined in the Remark \ref{Efficiency}. Let $\tilde{u}=u_{T_{\bullet}}$,
suppose that $(y,z,u)$ and $(y_{T_{\bullet}},z_{T_{\bullet}},u_{T_{\bullet}})$ are the solutions of the optimal control problem of (\ref{lianxuy})-(\ref{var_con}) and (\ref{lisan}), respectively. Then we have the error estimator $\mathcal{\eta}_{\mathcal{OCP}}$ satisfied the following lower bound
 \begin{eqnarray*}
\mathcal{\eta}_{\mathcal{OCP}}^2\leq C \|\mathbf{e}\|^2 + \sum\limits_{T\in T_{\bullet}}h_{\bullet}^{1-2\epsilon} \|y-y_{T_{\bullet}}\|^2_{H^{s+1/2-\epsilon}(\Omega^3_{\bullet}(T))}+\sum\limits_{T\in T_{\bullet}}h_{\bullet}^{1-2\epsilon} \|z-z_{T_{\bullet}}\|^2_{H^{s+1/2-\epsilon}(\Omega^3_{\bullet}(T))}.
     \end{eqnarray*}

\end{Remark}

\section{ Numerical\ results}
In this section, we use the following standard adaptive algorithm to optimize the convergence rate for numerical examples, which employs D$\rm{\ddot{o}}$rfler's marking criterion to designate elements for refinement.\\

\begin{center}\begin{tikzpicture}[every node/.style={rectangle, draw}]
    \node (A) at (-3,0) {\(\displaystyle \rm{Solve}\)};
        \node (B) at (0,0) {\(\displaystyle \rm{Estimate}\)};
    \node (C) at (3,0) {\(\displaystyle \rm{Mark}\)};
    \node (D) at (6,0) {\(\displaystyle \rm{Refine}\)};
    \draw[->] (A) -- (B);
    \draw[->] (B) -- (C);
    \draw[->] (C) -- (D);
\end{tikzpicture}
\end{center}

We use the residual error estimator $\mathcal{\eta}_{\mathcal{OCP}}$ to drive the adaptive mesh refinement and introduce the effectivity index $E_{\rm{eff}}:=\mathcal{\eta}_{\mathcal{OCP}}/\|\mathbf{e}\|_{\Omega}$.

\begin{algorithm}
\caption{Design\ of\ the\ AFEMs:}
 $\star\ \rm{Solve1}$: Initial mesh ${\mathcal{T}_{{\bullet}_{0}}}$, desired state $y_d$, external source $f$, constraints $a$ and $b$, regularization parameter $\lambda$.  Set $k=0$ and use the projection gradient algorithm to solve $(y_{{\mathcal{T}_{{\bullet}_k}}}, z_{{\mathcal{T}_{{\bullet}_k}}}, u_{{\mathcal{T}_{{\bullet}_k}}})\in\mathbb{V}_{T_{\bullet k}}\times \mathbb{V}_{T_{\bullet k}}\times U_{ad,h}(U_{ad}).$

 $\star\ \rm{Estimate}$: Compute the local error indicator $\mathcal{\eta}_{\mathcal{OCP},T}.$

 $\star\ \rm{Mark}$: Given a parameter $0\leq\theta\leq1$; Construct a minimal subset $\mathcal{M}_k\subset{\mathcal{T}_{{\bullet}_k}}$ such that
  $\sum\limits_{T\in\mathcal{M}_k}\mathcal{\eta}_{\mathcal{OCP},T}\geq \theta\mathcal{\eta}_{\mathcal{OCP}}.$ Mark all the elements in $\mathcal{M}_k.$

 $\star\ \rm{Refine}$: We bisect all the elements $T\in\mathcal{T}_{{\bullet}_k}$ that are contained in
$\mathcal{M}_k$ with the newest-vertex bisection method and create a new mesh
$\mathcal{T}_{{\bullet}_{k+1}}$. Define
$\mathcal{M}_{k+1}=\rm{refine}(\mathcal{M}_k).$

$\star\ \rm{Solve2}$: Construct the finite element space $\mathcal{T}_{{\bullet}_{k+1}}$ and solve $(y_{{\mathcal{T}_{{\bullet}_{k+1}}}}, z_{{\mathcal{T}_{{\bullet}_{k+1}}}}, u_{{\mathcal{T}_{{\bullet}_{k+1}}}})\in\mathbb{V}_{T_{\bullet k+1}}\times \mathbb{V}_{T_{\bullet k+1}}\times U_{ad,h}(U_{ad})$.

$\star\ \rm{End\ loop}$: If  $\mathcal{\eta}_{\mathcal{OCP},T}\leq \mathrm{Tol}_{\mathrm{space}}$, otherwise, set $k=k+1$ and go to Step $ \rm{Estimate}$.
\end{algorithm}

\begin{algorithm}
\caption{Projection gradient algorithm}
 Start with the mesh $\mathcal{T}_{{\bullet}_{t}}$ with mesh size $h_{\bullet}$.
Given the initial value $ u^0_{\mathcal{T}_{{\bullet}_{t}}}$,  and a tolerance
   $\mathrm{Tol}_{\mathrm{space}}>0$.

$\mathbf{While}$\ $error>\mathrm{Tol}_{\mathrm{space}}$

   $\mathbf{1.}$ \
   Solving the state equation in  (\ref{lisan}) to get state variable $y_{\mathcal{T}_{{\bullet}_{t}}}$;

  $\mathbf{2.}$  \ Solving the adjoint state equation in  (\ref{lisan})  to obtain adjoint state variable $z_{\mathcal{T}_{{\bullet}_{t}}}$;

  $\mathbf{3.}$ \ Compute the associted subgradient and control variable\\
$[\rm{Full\ discrete}]:\ u^{new}_{\mathcal{T}_{{\bullet}_{t}}}=\min\{ b, \max\{a,(\lambda|T|)^{-1}\sum\limits_{T\in\mathcal{T}_{\bullet_{t}}}\int_{T}y_{\mathcal{T}_{{\bullet}_{t}}}z_{\mathcal{T}_{{\bullet}_{t}}}dx \} \}$\\
$[\rm{semi-discrete}]:\ u^{new}_{\mathcal{T}_{{\bullet}_{t}}}=\min\{ b, \max\{a,\lambda^{-1}y_{\mathcal{T}_{{\bullet}_{t}}}z_{\mathcal{T}_{{\bullet}_{t}}} \} \}.$\\
$\mathbf{4.}$  Calculate the error:
 $ error=norm(u^0_{\mathcal{T}_{{\bullet}_{t}}}-u^{new}_{\mathcal{T}_{{\bullet}_{t}}},inf).$

$\mathbf{5.}$\ Update the control variable $u^0_{\mathcal{T}_{{\bullet}_{t}}}=u^{new}_{\mathcal{T}_{{\bullet}_{t}}}.$

$\mathbf{End\ While}$
\end{algorithm}

We now present two numerical examples that demonstrate how completely discrete and semi-discrete methods perform. In the first example, we consider a problem with an exact solution. Using the projection formula (\ref{u}), the state, and adjoint equations  (\ref{lianxuy})-(\ref{lianxuz}), we can compute the exact optimal control variable, the desired state $y_d$, and the source term $f$.
In the second example, we consider a problem where the solution in a square domain is unknown. We specify the desired state $y_d$ and the source term $f$.
\begin{Example}\label{exm1}
We set $\Omega=B(0,1)$, $a=0.4$, $b=1.5$, $\lambda=0.1$ and the exact solutions are as follows:
\begin{align*}
y=z=\frac{2^{-2s}(1-|x|^2)^{s}}{\Gamma(1+s)^2},\ u=\Pi_{[a,b]}(\lambda^{-1}zy).
\end{align*}
\end{Example}
In the context of variational discretization, Figure \ref{mesh1} presents the initial mesh and the final refinement mesh when parameterized with $s=0.25$ and $\theta=0.7$. It is evident that the mesh is concentrated along the boundaries where singularities are located. Numerical solutions of the state and control variables for $s=0.75$ and $\theta=0.5$ are depicted in Figure \ref{num1}.
 \begin{figure}[!htbp]
\centering
\label{2a}
\includegraphics[width=6.5cm]{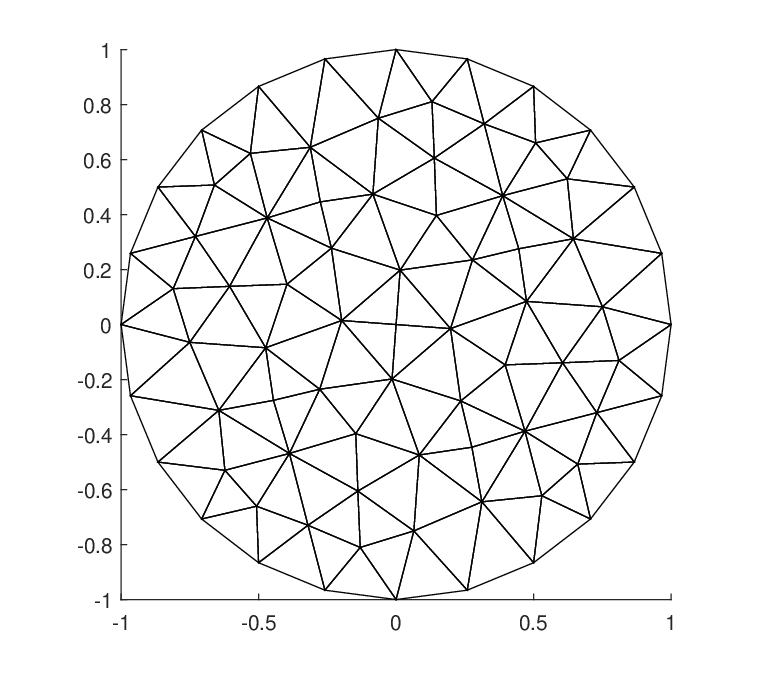}
\hspace{-0.01mm}
\label{2b}
\includegraphics[width=6.5cm]{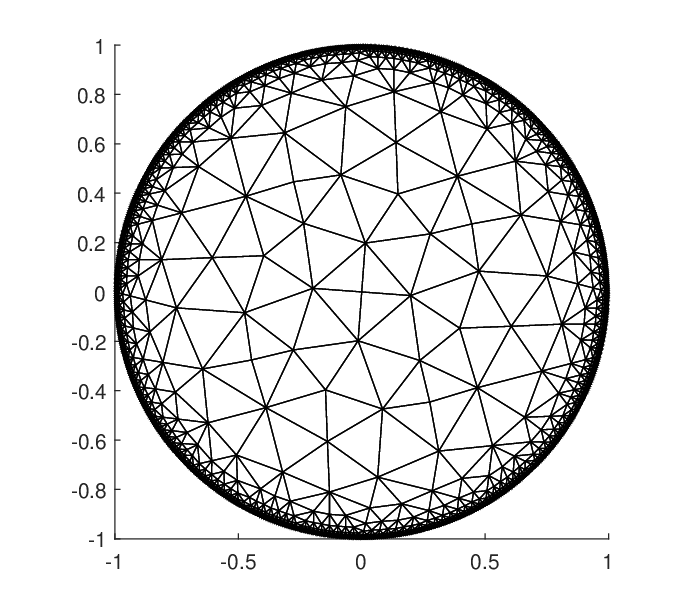}
\caption{The initial mesh (left)  and the final refinement mesh (right) with $s=0.25$ on the circle.}
\label{mesh1}
\end{figure}
\begin{figure}[!htbp]
\centering
\label{2a}
\includegraphics[width=6.5cm]{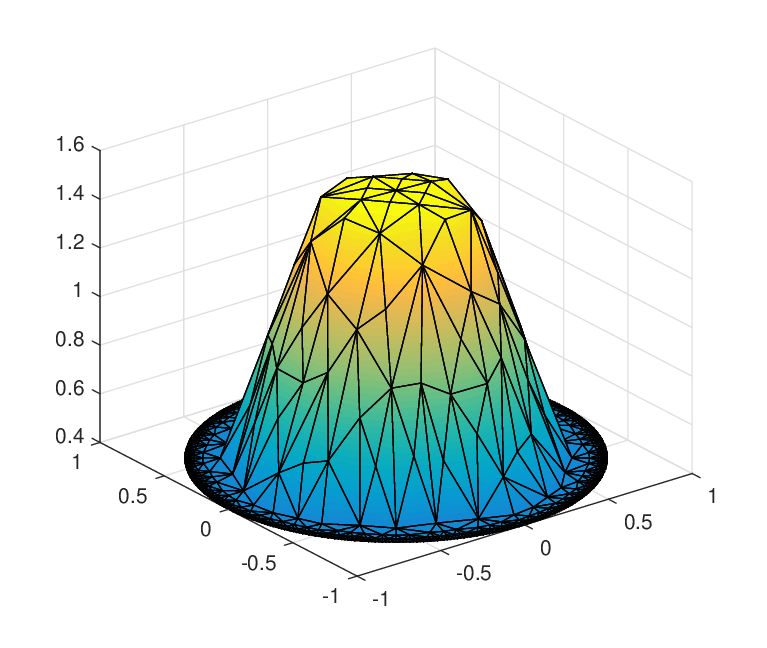}
\hspace{-0.01mm}
\label{2b}
\includegraphics[width=6.5cm]{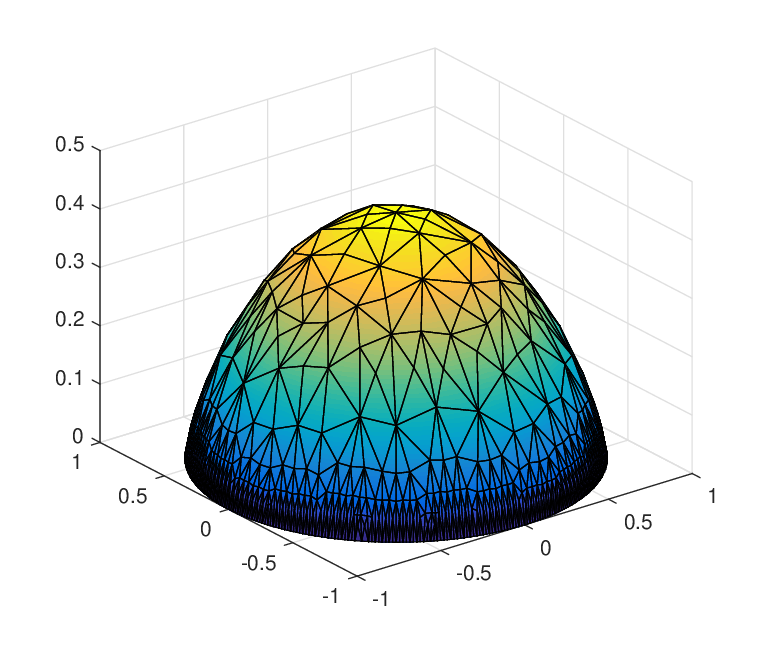}
\caption{The profiles of the numerically computed  control and state with $s=0.75,\ \theta=0.5$.}
\label{num1}
\end{figure}

To elucidate the effectiveness of the adaptive finite element method in solving optimal control problems, Figure \ref{error1} (left) illustrates the convergence behavior of the errors and the effectivity indices  in the optimal control, state, and adjoint state when  $s=0.75$ and $\theta=0.5$.  In Figure \ref{error1} (right), we showcase the convergence behavior of the error indicators $\mathcal{\eta}_{\mathcal{OCP}}$ and the error estimators $\mathcal{\eta}_{\mathcal{Y}}(y_{\mathcal{T}_{\bullet}}),\ \mathcal{\eta}_{\mathcal{Z}}(z_{\mathcal{T}_{\bullet}})$.  Under uniform refinement, the convergence rate of the error approximates a straight line with a slope of $-1/4$. However, under adaptive refinement, the convergence rate of the state and adjoint errors in the $\|\cdot\|_{\widetilde{H}^{s}(\Omega)}$ norm is $N^{-1/2}$, and the convergence order of the control error in the $L^2$ norm can reach $N^{-1}$, consistent with the desired optimal convergence rate.
  For $s=0.25, \theta=0.7$, Figure \ref{error2} illustrates the convergence of errors and effectivity indices in the optimal control, state, and adjoint state on both adaptive and uniform refinement grids and the convergence of error indicators and error estimators.

 \begin{figure}[!htbp]
\centering
\label{2a}
\includegraphics[width=6.6cm]{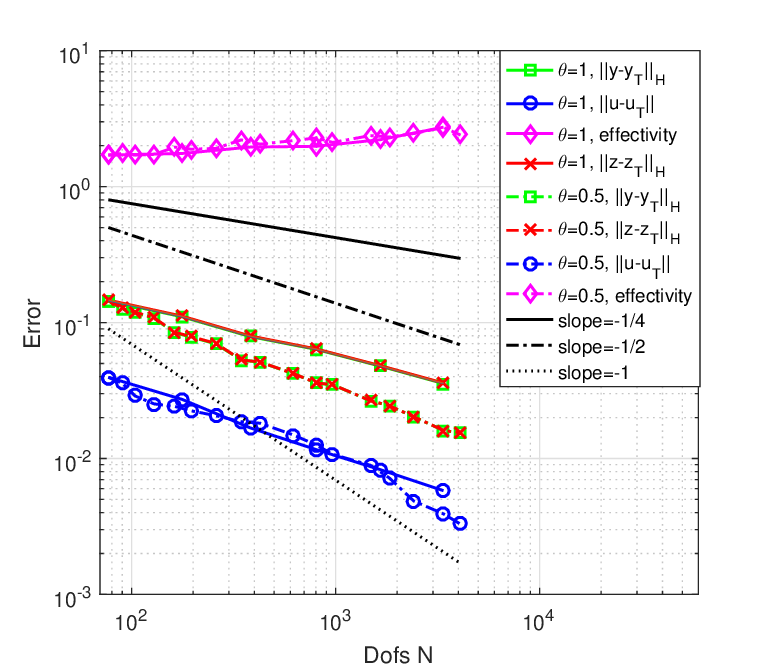}
\hspace{-0.01mm}
\label{2b}
\includegraphics[width=6.5cm]{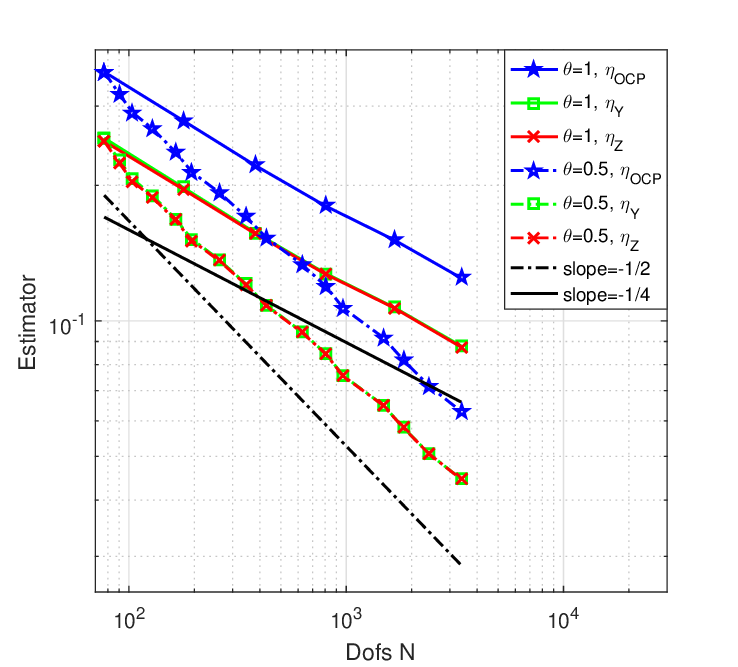}
\caption{The convergent behaviors of the errors and effectivity indices (left) and the convergent behaviors of the error indicators and error estimators (right) with $s=0.75$.}
\label{error1}
\end{figure}
\begin{figure}[!htbp]
\centering
\label{2a}
\includegraphics[width=6.5cm]{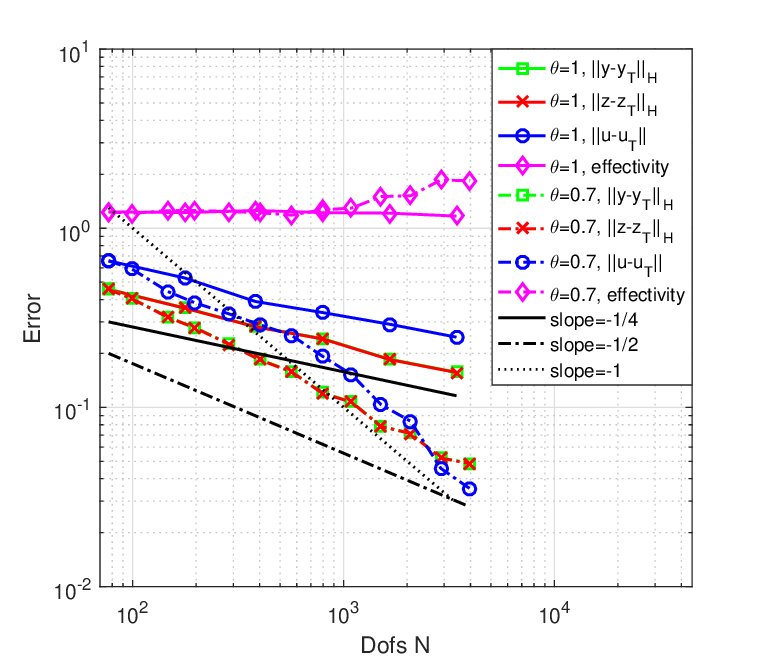}
\hspace{-0.01mm}
\label{2b}
\includegraphics[width=6.5cm]{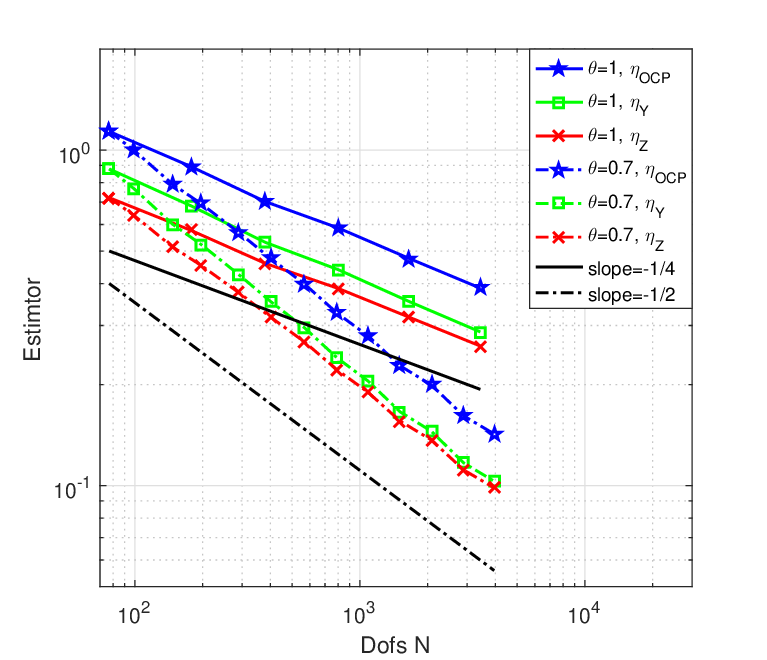}
\caption{The convergent behaviors of the errors and effectivity indices (left) and the convergent behaviors of the error indicators and error estimators (right) with $s=0.25$ respectively.}
\label{error2}
\end{figure}

In the context of full discretization,  Figure \ref{mesh2} shows the refinement mesh resulting from 15 adaptive iterations and the profile of the numerically computed control for $s=0.75$ and $\theta=0.5$.  The numerical solution plots for the control variables show that the interior of the region presents piecewise constants, thus leading to significant region discontinuities. After mesh refinement at boundary singularities, the interior of the mesh is further encrypted. To investigate the role of error estimators, we conducted 10 iterations of grid refinement using control variable error estimators. Figure \ref{control} shows that the refined regions indeed lie in the interior, indicating that the boundary refinement in Figure \ref{mesh2} (left) is caused by the singularity of the fractional-order operator, while the interior refinement is due to the discontinuity of the control variable.

\begin{figure}[!htbp]
\centering
\label{2a}
\includegraphics[width=6.3cm]{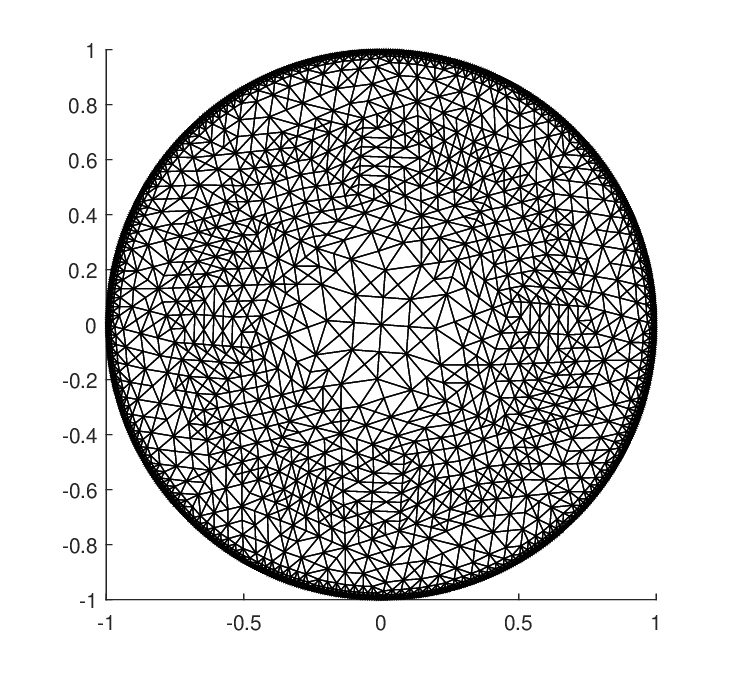}
\hspace{-0.01mm}
\label{2b}
\includegraphics[width=6.6cm]{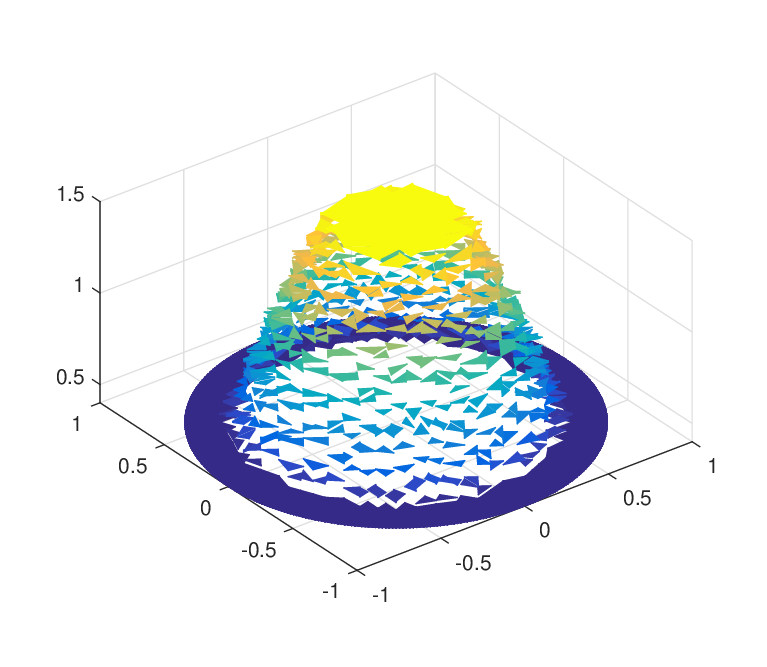}
\caption{The mesh after 15 adaptive steps with $s=0.75$ and the profile of the numerically computed control.}
\label{mesh2}
\end{figure}
\begin{figure}[!htbp]
\centering
\label{2a}
\includegraphics[width=6.5cm]{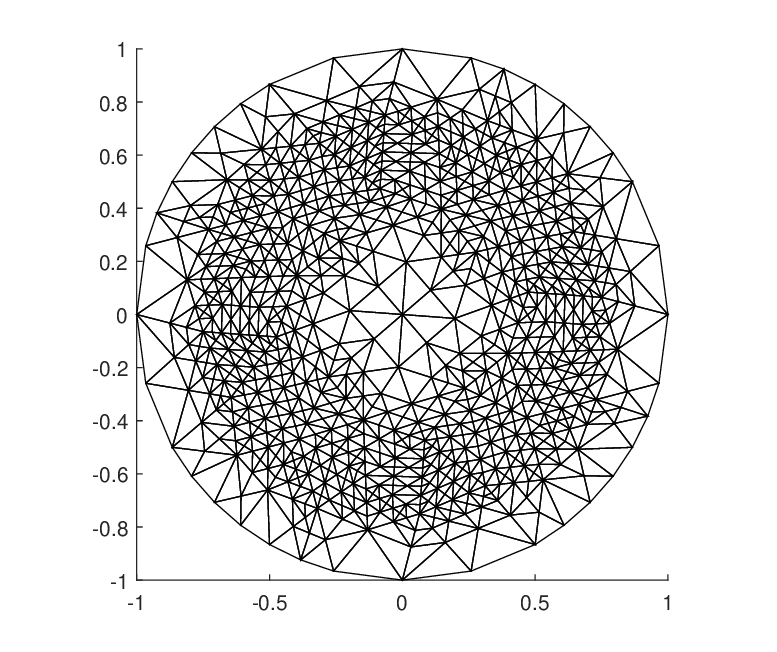}

\caption{The mesh after 10 adaptive refinements of the control error estimator with $s=0.75$.}
\label{control}
\end{figure}
Likewise, the convergence behavior of errors, error estimators on both adaptive and uniform refinement grids is presented in Figure \ref{error3}, demonstrating the attainment of the optimal convergence rate.  Since $u_{\mathcal{T}_{\bullet}}$ is implicitly discretized with piecewise constant functions, the state, adjoint and control variables have convergence order $N^{-1/2}$.

\begin{figure}[!htbp]
\centering
\label{2a}
\includegraphics[width=6.5cm]{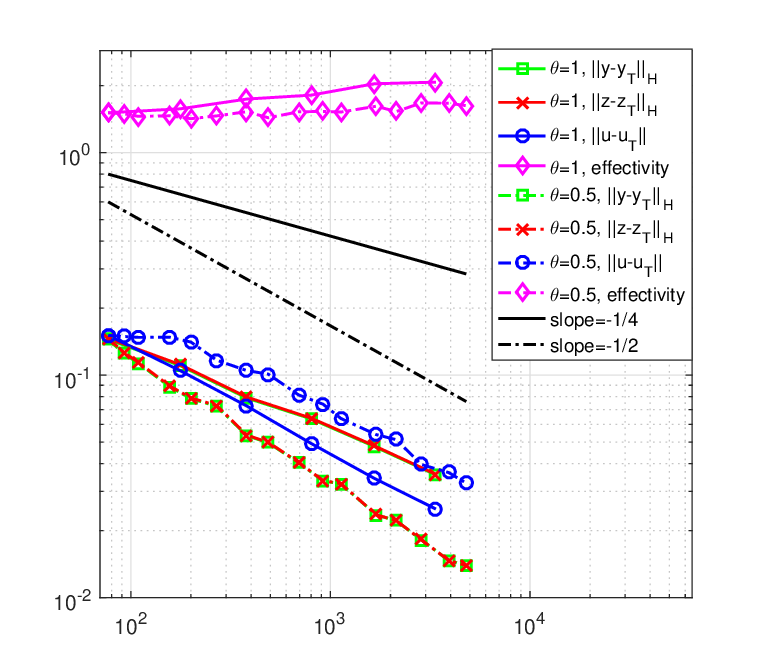}
\hspace{-0.01mm}
\label{2b}
\includegraphics[width=6.5cm]{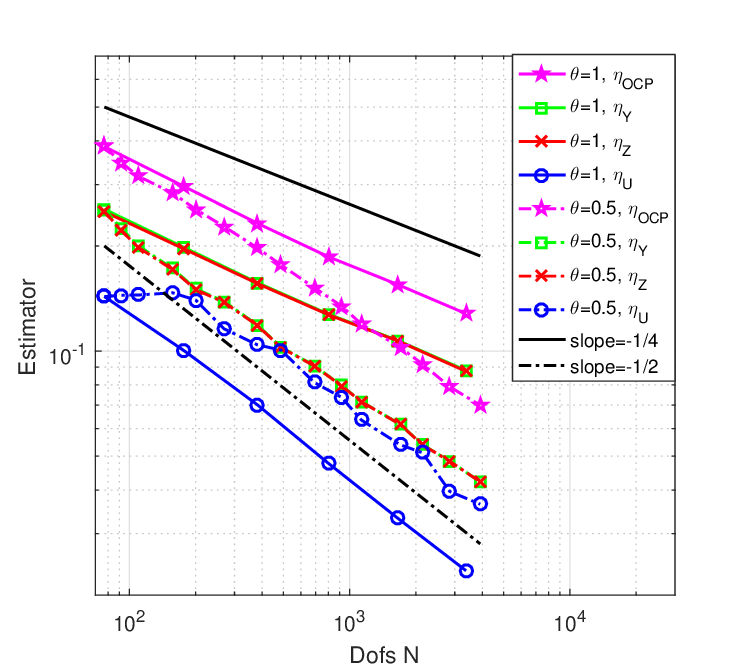}
\caption{The convergent behaviors of  the errors, error and effectivity indices with $s=0.75$.}
\label{error3}
\end{figure}
The refinement mesh resulting from 13 adaptive iterations and the profile of the numerically computed control for $s=0.25$ and $\theta=0.7$ are provided in Figure \ref{num2}. In this case, the interior of the numerical solution of the control variable is all controlled by the upper bound of the constrained set, and only the boundary is the piecewise  constant, which means that the numerical solution is very smooth inside the region. At the same time, we know that the solution is singular on the boundary, so the mesh encryption will always be refined at the boundary and not inside, which is consistent with results in Figure \ref{num2}.

In Figure \ref{error4}, the convergence orders of  errors, error indicators and error estimators are presented for different $\theta$ values. Because the discontinuity and singularity of the control variable are on the boundary, as the mesh is continuously refined on the boundary, it will achieve the same result as the variational discretization. Therefore, the error of the control variable can reaches $N^{-1}$, as shown in Figure \ref{error4}.

\begin{figure}[!htbp]
\centering
\label{2a}
\includegraphics[width=6.5cm]{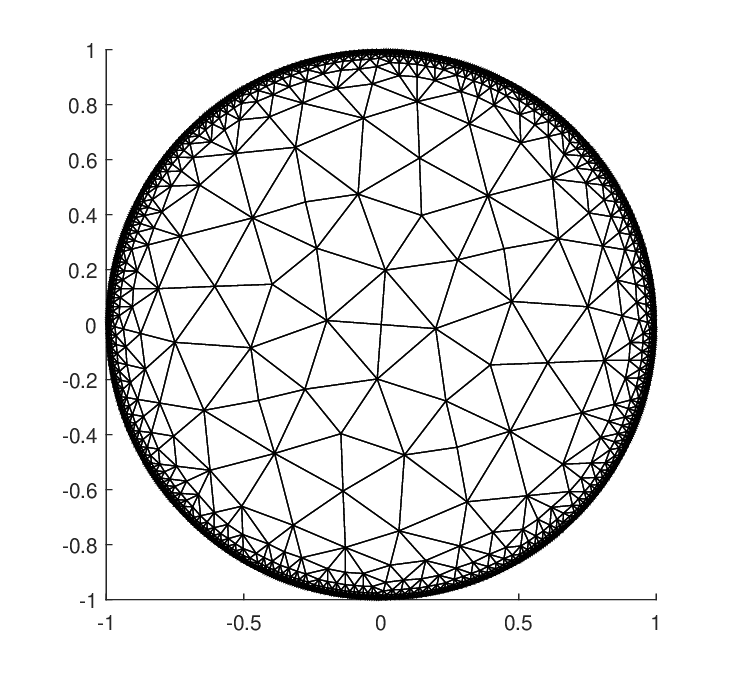}
\hspace{-0.01mm}
\label{2b}
\includegraphics[width=6.6cm]{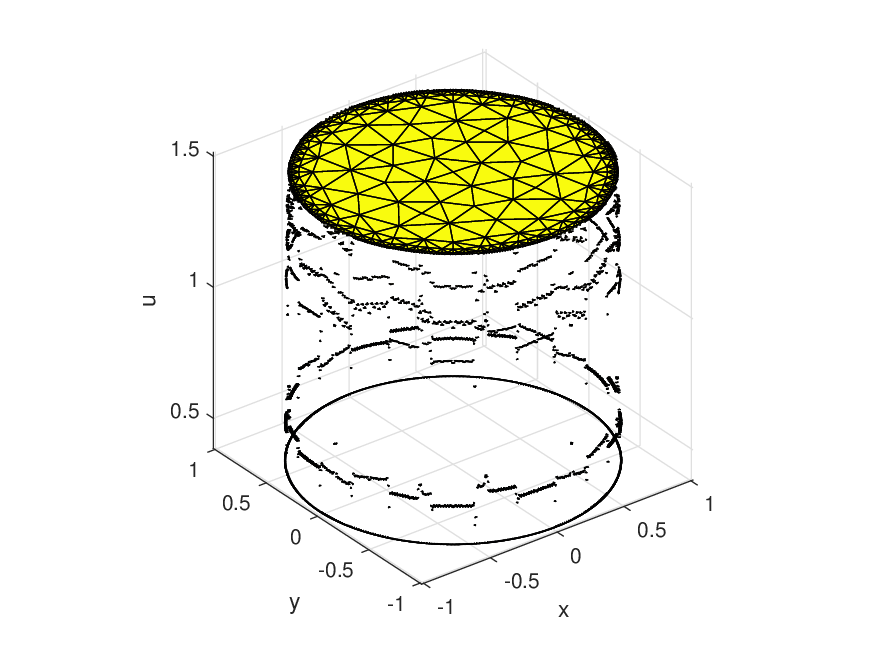}
\caption{The mesh after 13 adaptive steps with $s=0.25$ and the profile of the numerically computed control.}
\label{num2}
\end{figure}

\begin{figure}[!htbp]
\centering
\label{2a}
\includegraphics[width=6.5cm]{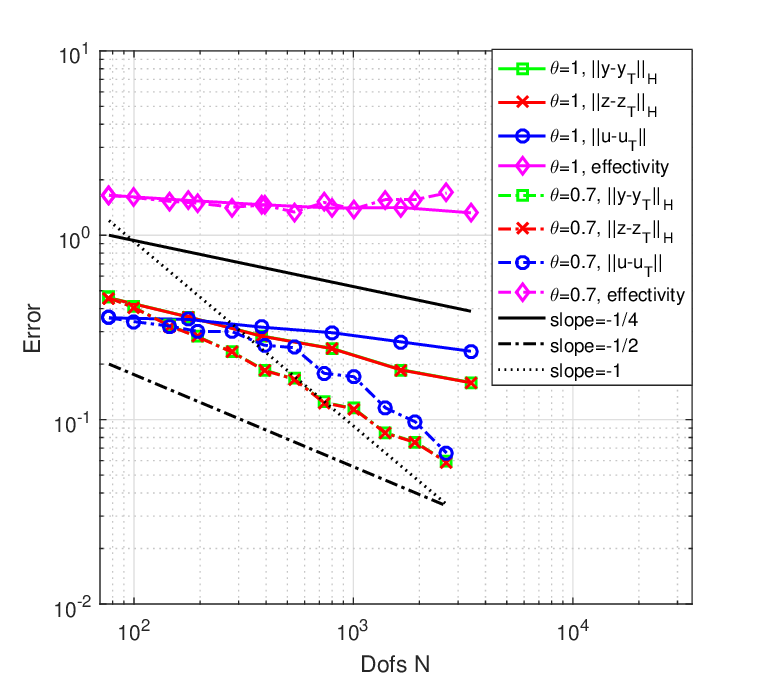}
\hspace{-0.01mm}
\label{2b}
\includegraphics[width=6.5cm]{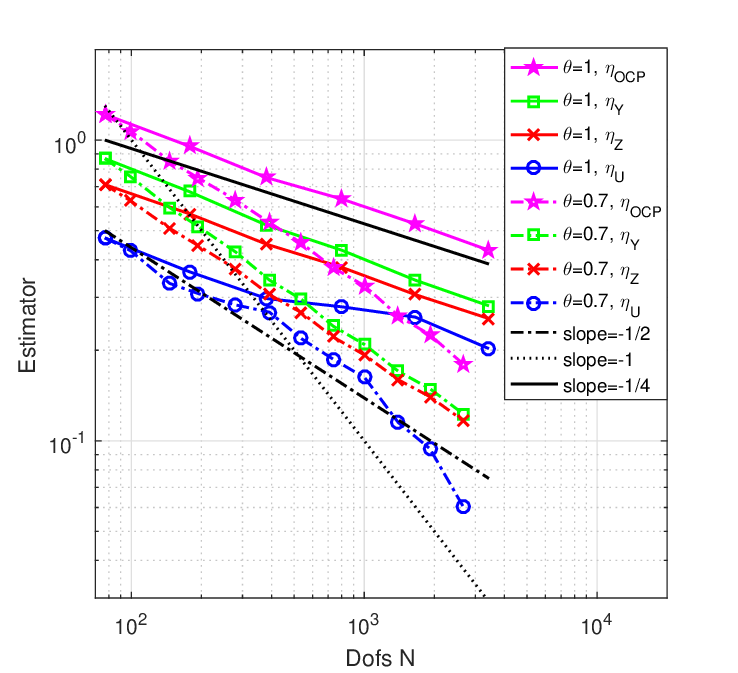}
\caption{The convergent behaviors of the errors, error indicators and error estimators with $s=0.25$.}
\label{error4}
\end{figure}

When considering the parameter values  $a=0.4, b=0.7, \lambda=1,$ $s=0.25$ and $\theta=0.7$, as illustrated in Figure \ref{fullab}, it becomes evident that beyond the adaptive refinement of the boundary, an interior refinement process is initiated. This refinement is necessitated by the presence of piecewise constant variations within the numerical solution of the control variable. Notably, the disparity between the outcomes presented in Figure \ref{fullabestimator} and Figure \ref{error4} arises from the influence of this piecewise constant component on the convergence behavior of the control variable. Specifically, the order of convergence is no longer characterized as $N^{-1}$, but rather aligns with that of the state and adjoint variables, yielding $N^{-1/2}$.

\begin{figure}[!htbp]
\centering
\label{2a}
\includegraphics[width=6.5cm]{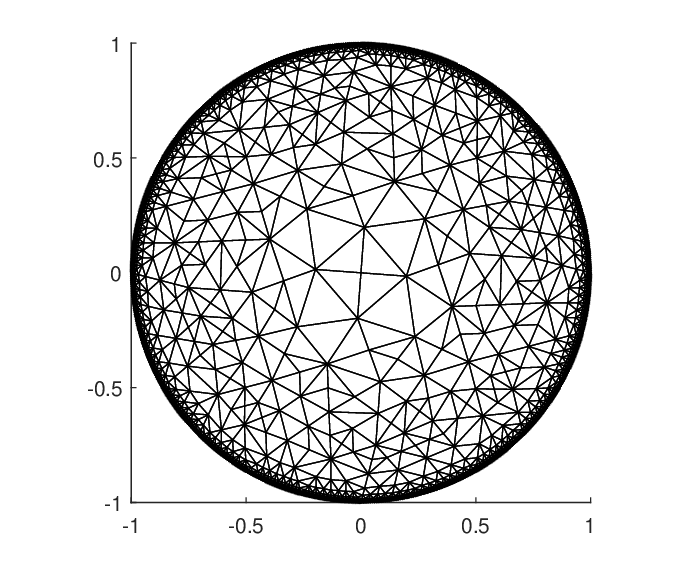}
\hspace{-0.01mm}
\label{2b}
\includegraphics[width=6.6cm]{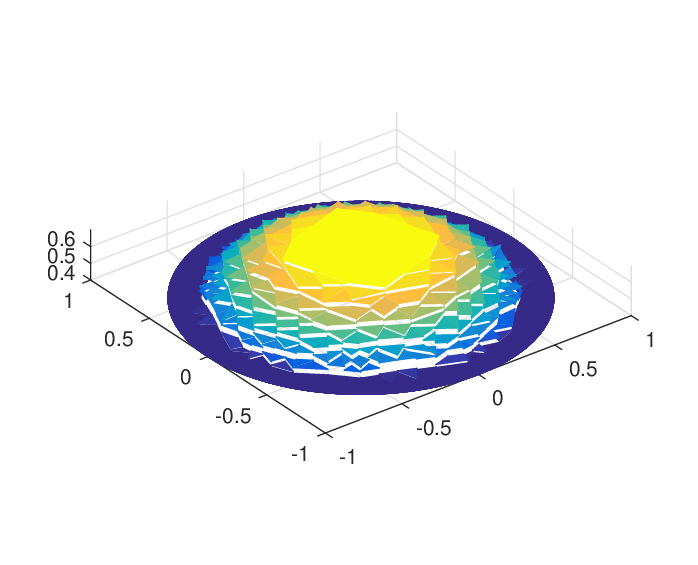}
\caption{The mesh after 13 adaptive steps with $s=0.25$ and the profile of the numerically computed control.}
\label{fullab}
\end{figure}

\begin{figure}[!htbp]
\centering
\label{2a}
\includegraphics[width=6.5cm]{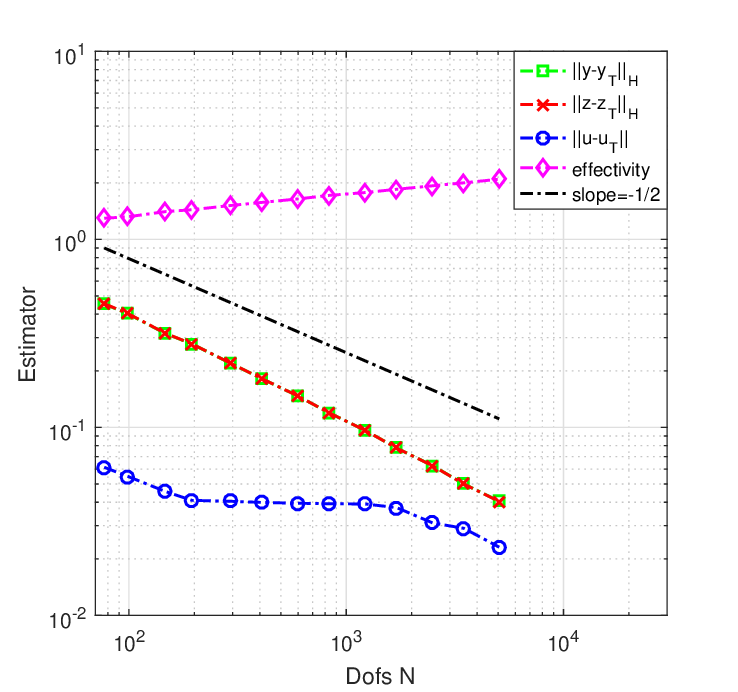}
\hspace{-0.01mm}
\label{2b}
\includegraphics[width=6.5cm]{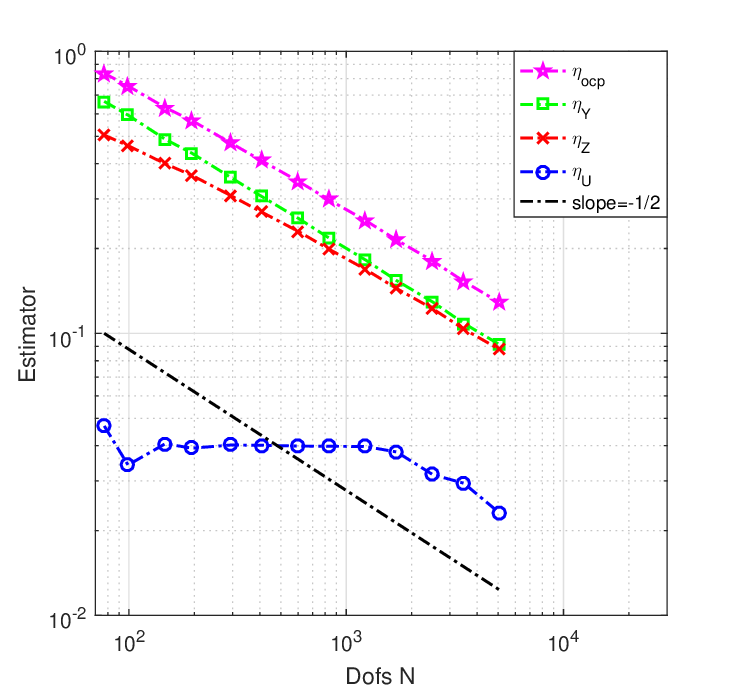}
\caption{The convergent behaviors of the errors, error indicator and error estimators.}
\label{fullabestimator}
\end{figure}

\begin{Example}\label{exm:2}
In the second example we consider an optimal control problem with $f=-4,\ y_d=4$. We set $\Omega=(-1,1)^{2}$, $\lambda=1$, $a=0.5$, $b=1.5$.
\end{Example}

In the context of variational discretization,  Figure \ref{mesh3} (left) provides the initial mesh for the case of $s=0.25$ and $\theta=0.7$, while the right side depicts the final refinement mesh after 13 adaptive iterations. The primary refinement behavior is observed to occur exclusively along the boundaries of the entire square domain.  This observation suggests that the estimators effectively capture the singularities of the exact solution along the entire boundary, thus guiding the mesh refinement process. The numerically computed optimal state and adjoint state profiles are exhibited in Figure \ref{num3}.
 \begin{figure}[!htbp]
\centering
\label{2a}
\includegraphics[width=6.5cm]{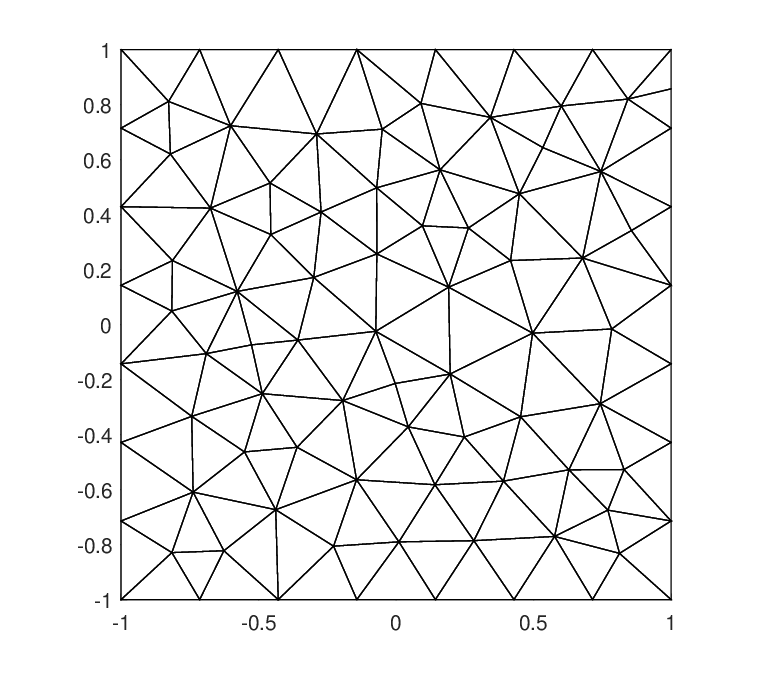}
\hspace{-0.01mm}
\label{2b}
\includegraphics[width=6.5cm]{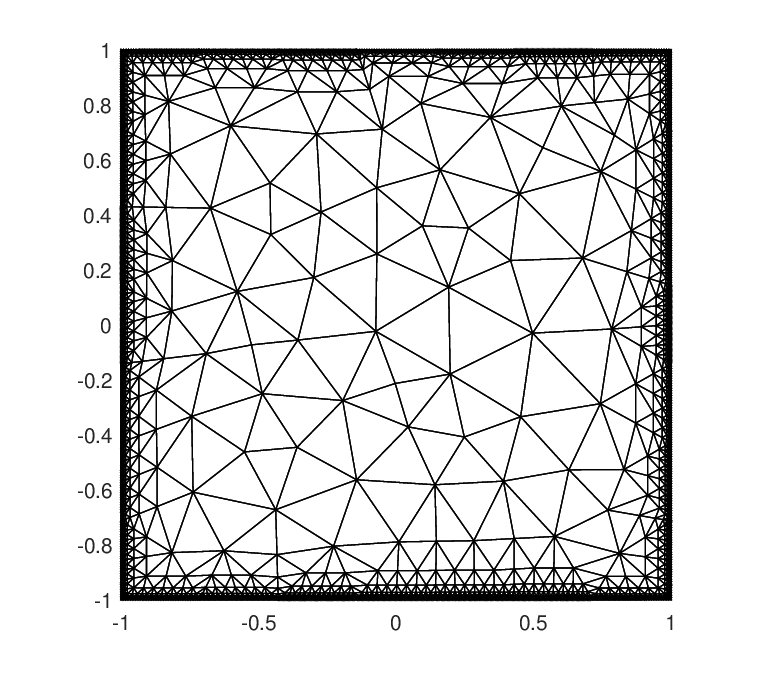}
\caption{The initial mesh (left)  and the final refinement mesh (right) with $s=0.25$ on the square.}
\label{mesh3}
\end{figure}
\begin{figure}[!htbp]
\centering
\label{2a}
\includegraphics[width=6.5cm]{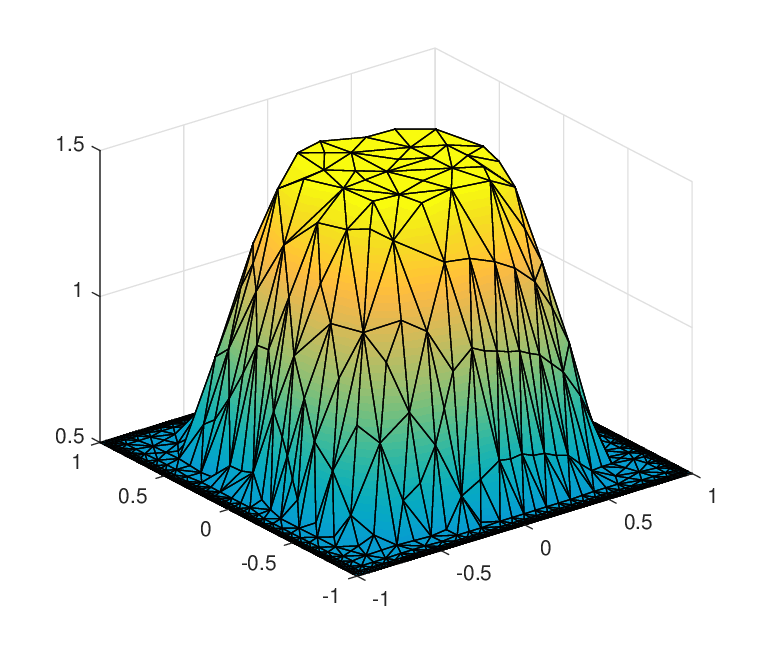}
\hspace{-0.01mm}
\label{2b}
\includegraphics[width=6.5cm]{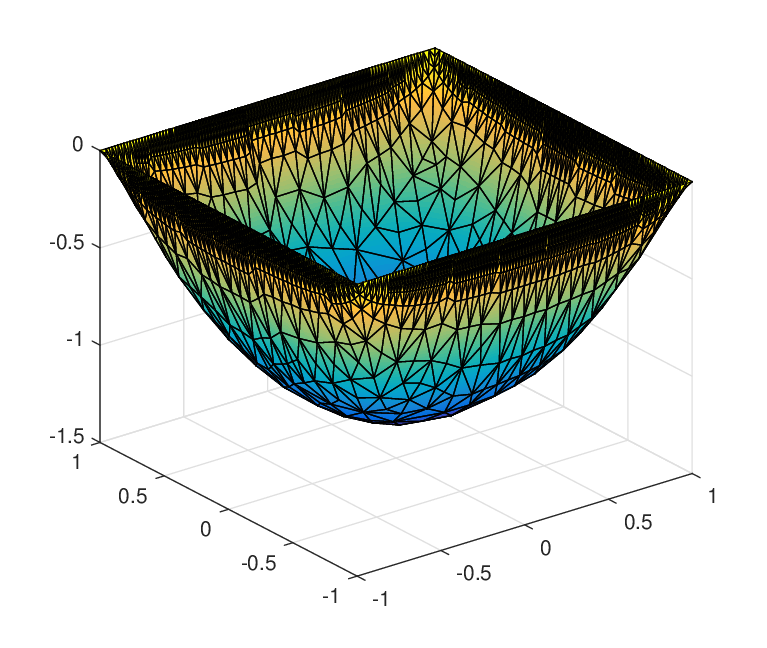}
\caption{The profiles of the numerically computed  control and state with $s=0.75$.}
\label{num3}
\end{figure}

 It can be seen from the Figure \ref{error5} that the error indicators  $\mathcal{\eta}_{\mathcal{OCP}}$ and the error estimators   $\mathcal{\eta}_{\mathcal{Y}}(y_{\mathcal{T}_{\bullet}}),\   \mathcal{\eta}_{\mathcal{Z}}(z_{\mathcal{T}_{\bullet}}), \mathcal{\eta}_{\mathcal{U}}(u_{\mathcal{T}_{\bullet}})$ can reach the optimal convergence order for all the values of the parameter $\theta$ considered.

\begin{figure}[!htbp]
\centering
\label{2a}
\includegraphics[width=6.5cm]{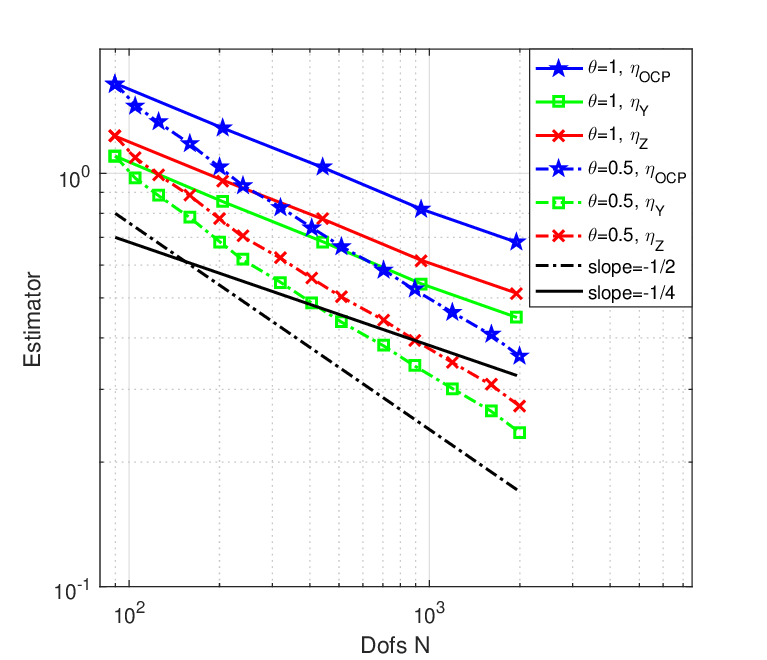}
\hspace{-0.01mm}
\label{2b}
\includegraphics[width=6.5cm]{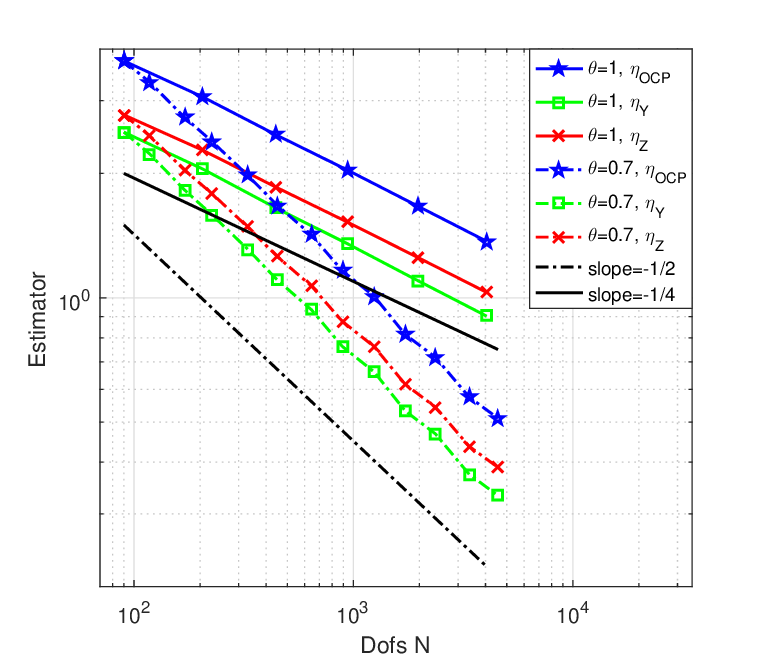}
\caption{The convergent behaviors of the error indicators and error estimators with $s=0.75$ (left) and $s=0.25$ (right) respectively on the square.}
\label{error5}
\end{figure}

 For the full discretization,  Figure \ref{mesh4} (left) provides  the final refinement mesh after 14 adaptive iterations for the case of $s=0.75$ and $\theta=0.5$, while the right side depicts the final refinement mesh after 13 adaptive iterations with $s=0.25$ and $\theta=0.7$. The profiles of the numerical control variables are presented in Figure \ref{num4}. In Figures \ref{error6}, we provide experimental convergence rates of the a posteriori error estimators for the optimal control, state, and adjoint state.  Across all considered values of $s$ and $\theta$, we consistently observe optimal experimental convergence rates. The analysis of the convergence results for the control variables is the same as in Example \ref{exm1}.
  \begin{figure}[!htbp]
\centering
\label{2a}
\includegraphics[width=6.5cm]{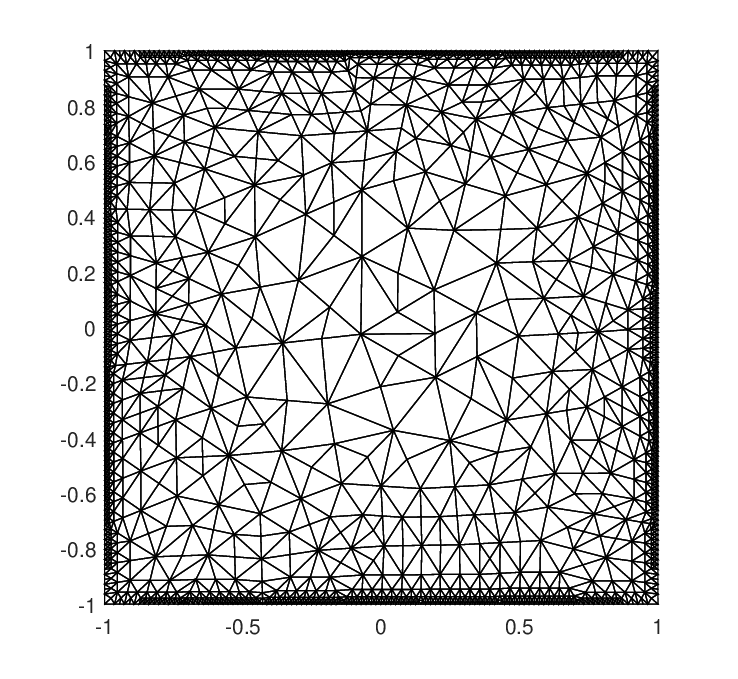}
\hspace{-0.01mm}
\label{2b}
\includegraphics[width=6.5cm]{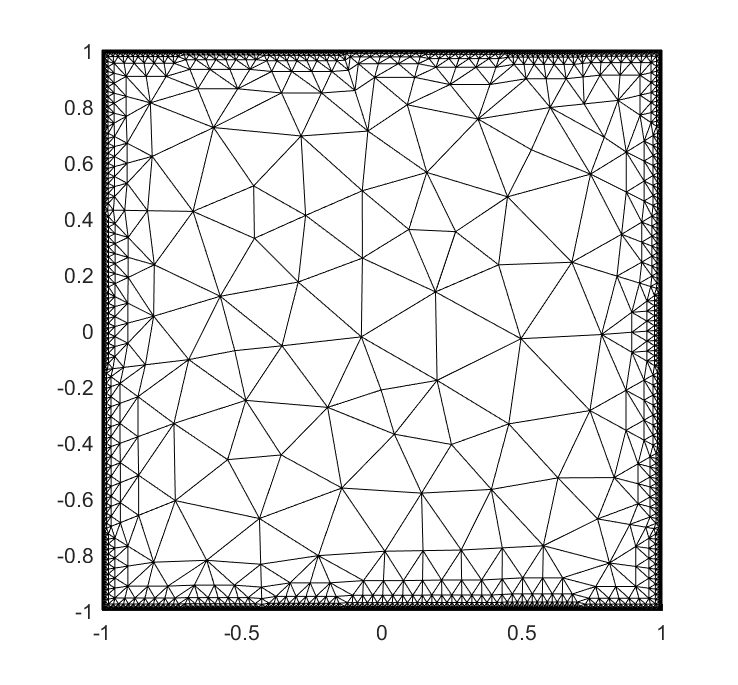}
\caption{The meshes after 14 adaptive steps with $s=0.75$ (left) and 13 adaptive steps with $s=0.25$ (right) on the square.}
\label{mesh4}
\end{figure}
 \begin{figure}[!htbp]
\centering
\label{2a}
\includegraphics[width=6.5cm]{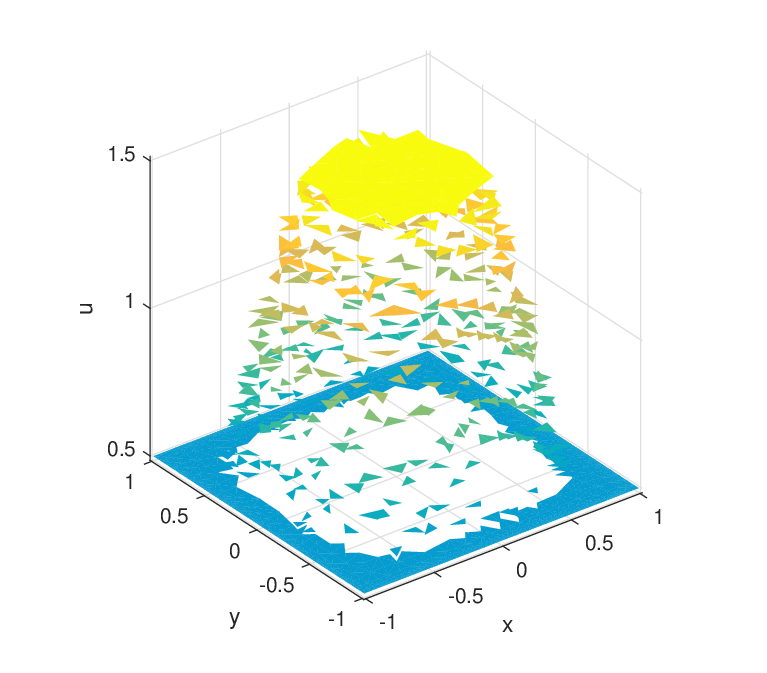}
\hspace{-0.01mm}
\label{2b}
\includegraphics[width=6.5cm]{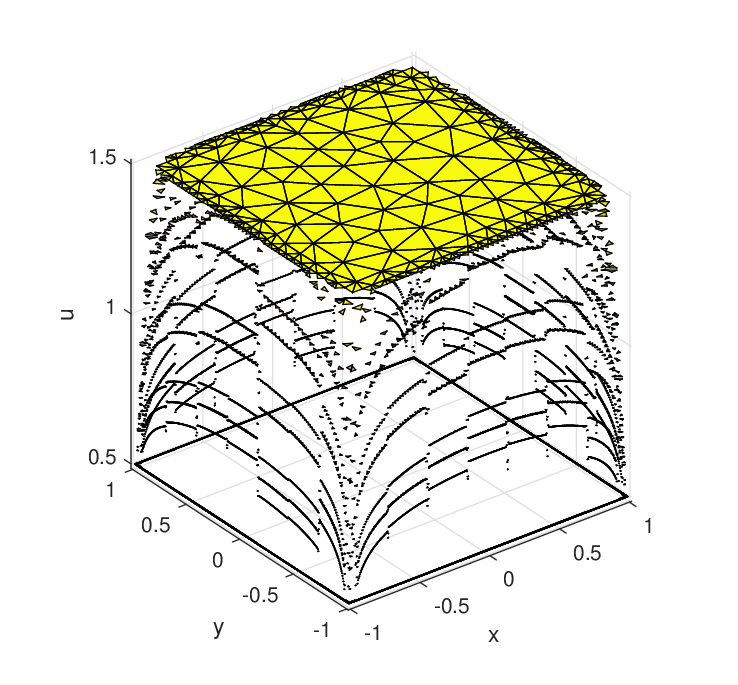}
\caption{The profiles of the numerically computed  control  with $s=0.75$ and control  with $s=0.25$.}
\label{num4}
\end{figure}
\begin{figure}[!htbp]
\centering
\label{2a}
\includegraphics[width=6.5cm]{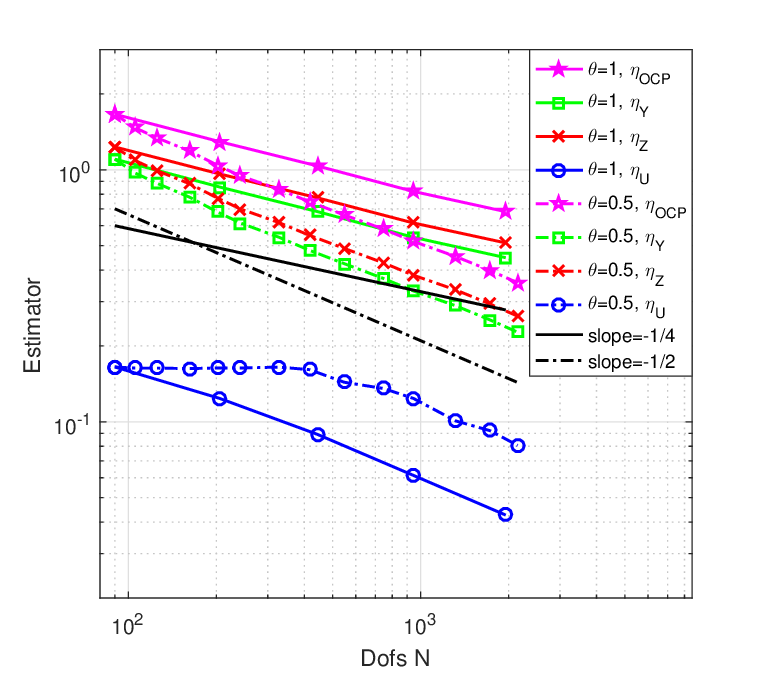}
\hspace{-0.01mm}
\label{2b}
\includegraphics[width=6.5cm]{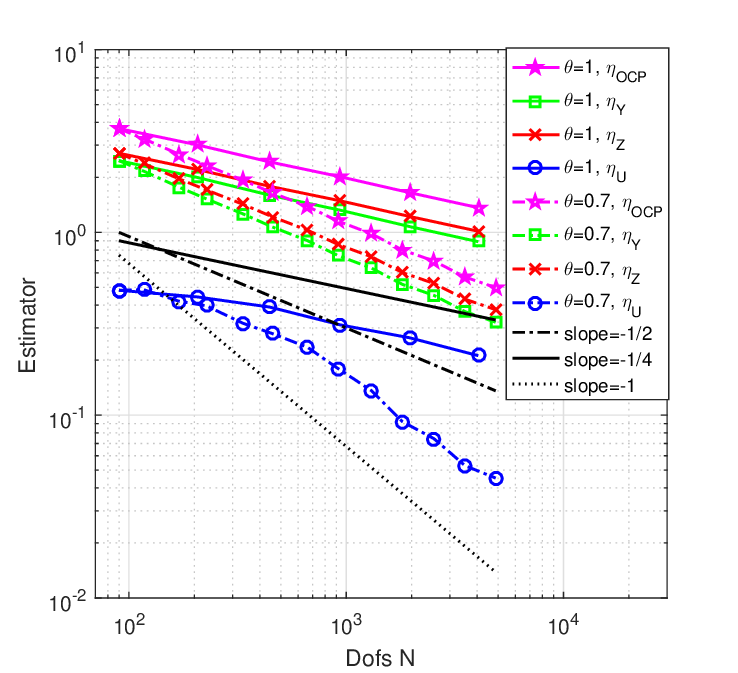}
\caption{Computational rates of convergence for the contributions $\mathcal{\eta}_{\mathcal{Y}}(y_{\mathcal{T}_{\bullet}}),\ \mathcal{\eta}_{\mathcal{Z}}(z_{\mathcal{T}_{\bullet}})$ and $ \mathcal{\eta}_{\mathcal{U}}(u_{\mathcal{T}_{\bullet}})$ of the computable and anisotropic a posteriori error estimator  $\mathcal{\eta}_{\mathcal{OCP}}$.}
\label{error6}
\end{figure}

\section*{Acknowledgements}

The work was supported by the National Natural
Science Foundation of China under Grant No. 11971276  and 12171287.

\end{document}